\documentclass[11pt,reqno]{amsart}
\usepackage[dvipsnames,usenames]{color}
\usepackage{hyperref}
\usepackage{graphicx}
\usepackage{enumerate, enumitem}
\usepackage{cleveref}
\usepackage[all]{xy}
\usepackage{scalerel}
\usepackage[top=1.8in, bottom=1.6in, left=1.5in, right=1.5in]{geometry}

\usepackage{amsmath}
\usepackage{amsfonts}
\usepackage{amssymb}
\usepackage{mathrsfs}
\usepackage{comment}
\newcommand{\Hom}{\operatorname{Hom}}

\renewcommand{\det}{\operatorname{det}}
\newcommand{\id}{\operatorname{id}}
\newcommand{\Hol}{\operatorname{Hol}}	
\newcommand{\N}{\mathbb{N}}
\newcommand{\Z}{\mathbb{Z}}

\newcommand{\R}{\mathbb{R}}

\newcommand{\A}{\mathscr{A}}

\newcommand{\su}{\mathfrak{su}}

\newcommand{\G}{\mathscr{G}}

\newcommand{\tr}{\text{tr}}

\makeatletter
\newcommand{\newreptheorem}[2]{%
\newtheorem*{rep@#1}{\rep@title}
\newenvironment{rep#1}[1]{%
\def\rep@title{#2 \ref*{##1}}%
\begin{rep@#1}}%
{\end{rep@#1}}}
\makeatother

\newtheorem{theorem}{Theorem}[section]
\newreptheorem{theorem}{Theorem}
\newtheorem{lemma}[theorem]{Lemma}
\newreptheorem{lemma}{Lemma}
\newtheorem{proposition}[theorem]{Proposition}
\newreptheorem{proposition}{Proposition}
\newtheorem{corollary}[theorem]{Corollary}
\newreptheorem{corollary}{Corollary}
\newtheorem{conjecture}{Conjecture}
\theoremstyle{definition}
\newtheorem{definition}[theorem]{Definition}
\theoremstyle{remark}
\newtheorem{remark}[theorem]{Remark}

\usepackage{scalerel} 
\newcommand{\smallcirc}{\scaleobj{0.6}{\circ}}
\author[Tye Lidman]{Tye Lidman}
\address{Department of Mathematics, North Carolina State University, Raleigh, NC 27607}
\email{tlid@math.ncsu.edu}

\author[Juanita Pinz\'on-Caicedo]{Juanita Pinz\'on-Caicedo}
\address {University of Notre Dame, Department of Mathematics, Notre Dame, IN 46556, USA.}
\email{jpinzonc@nd.edu}

\author[Raphael Zentner]{Raphael Zentner}
\address{Fakult\"at f\"ur Mathematik, Universit\"at Regensburg, 93040 Regensburg, Germany}
\email{raphael.zentner@mathematik.uni-regensburg.de}

\title[Toroidal homology spheres and $SU(2)$-representations]{Toroidal integer homology three-spheres have irreducible $SU(2)$-representations}
\date{}

\begin{document}

\begin{abstract}
We prove that if an integer homology three-sphere contains an embedded incompressible torus, then its fundamental group admits irreducible $SU(2)$-representations.  Our methods use instanton Floer homology, and in particular the surgery
exact triangle, holonomy perturbations, and a non-vanishing result due
to Kronheimer-Mrowka, as well as results about surgeries on cables due
to Gordon.
\end{abstract}
\maketitle

\section{Introduction}

%The study of closed 3-manifolds in terms of their fundamental group is a topic of great interest since 
%The most powerful of the standard invariants of algebraic topology for distinguishing 3-manifolds is the fundamental group. This determines all the homology groups of a closed orientable 3 manifold M .

The fundamental group is one of the most powerful invariants to distinguish closed three-manifolds. In fact, by Perelman's proof of Thurston's Geometrization conjecture \cite{Perelman1, Perelman2, Perelman3}, fundamental groups determine closed, orientable three-manifolds up to orientations of the prime factors and up to the indeterminacy arising from lens spaces. Prominently, the three-dimensional Poincar\'e conjecture, a special case of Geometrization, characterizes $S^3$ as the unique closed, simply-connected three-manifold.  %It is therefore natural to try to measure the non-triviality of the fundamental group of any other three-manifold. 
For a three-manifold with non-trivial fundamental group, it is then useful to quantify the non-triviality of the fundamental group. Since the Geometrization theorem implies that three-manifolds have residually finite fundamental groups \cite{Hempel}, this non-triviality can be measured by representations to finite groups. However, there is not a finite group $G$ such that every three-manifold group has a non-trivial homomorphism to $G$. Therefore, a more uniform measurement of non-triviality can be found in the following conjecture:

\begin{conjecture}[Kirby Problem 3.105(A), \cite{Kirby}]\label{conj:su2}
If $Y$ is a closed, connected, three-manifold other than $S^3$, then $\pi_1(Y)$ admits a non-trivial $SU(2)$-representation.  
\end{conjecture}

Note that this conjecture is equivalent to the statement that the fundamental groups of all integer homology three-spheres other than $S^3$ admit irreducible $SU(2)$-representations. Indeed, every three-manifold whose first homology group is non-zero admits non-trivial abelian representations to $SU(2)$. Moreover, lens spaces are examples of manifolds that admit non-trivial $SU(2)$-representations of their fundamental groups, but no irreducible ones. There are also three-manifolds with non-abelian fundamental group which do not admit irreducible representations \cite{Motegi}. However, for representations of perfect groups to $SU(2)$, non-triviality is equivalent to irreducibility.  

For comparison, the third author showed in \cite{Zentner} that \Cref{conj:su2} is true if one replaces $SU(2)$ with $SL_2(\mathbb{C})$. The reader may also relate \Cref{conj:su2} with characterizing the three-manifolds with simplest instanton or Heegaard Floer homologies.  One side of the L-space conjecture predicts that every prime integer homology three-sphere other than $S^3$ and the Poincar\'e homology three-sphere admits a co-orientable taut foliation.  This fact, together with the gauge-theoretic methods used by Kronheimer-Mrowka in \cite{KM_P}, would then imply \Cref{conj:su2}.  %Note that lens spaces are examples which cannot have irreducible representations, but they do have non-trivial reducible representations.   More generally, non-trivial first homology automatically yields a non-trivial abelian $SU(2)$- representation.  Therefore, we will restrict our attention to integer homology three-spheres.  However, in this case, irreducibility and non-triviality are equivalent.  

There are many families of integer homology three-spheres for which \Cref{conj:su2} has been established, such as those which are Seifert fibered (although the methods go back to Fintushel-Stern \cite{FintushelStern}, this can be found explicitly in \cite[Theorem 2.1]{Menagerie}), branched double covers of non-trivial knots with determinant 1 \cite[Theorem 3.1]{CornwellNgSivek} and \cite[Corollary 9.2]{Zentner_Simple}, $1/n$-surgeries on non-trivial knots in $S^3$ \cite{KM_Dehn}, those that are filled by a Stein manifold which is not a homology ball \cite{BaldwinSivek}, or for splicings of knots in $S^3$ \cite{Zentner}.  

%Recall that every prime summand of a closed three-manifold $M$ admits a JSJ torus decomposition \cite{JacoShalen, Johannson}. This means that there exists a minimal collection of incompressible tori in $M$ such that its complement in $M$ consists of connected pieces that are Seifert fibered or non-Seifert and atoroidal (i.e., do not admit an embedded non-boundary parallel $\pi_1$-injective torus). A consequence of Perelman's proof of Thurston's geometrization conjecture \cite{Perelman1, Perelman2, Perelman3} is that the atoroidal pieces admit a finite volume hyperbolic structure.  \tl{We want to say non-Seifert fibered atoroidal pieces?}

It follows again from Geometrization that there are three (non-disjoint) types of prime integer homology three-spheres: Seifert fibered, hyperbolic, and toroidal ones. We remark that although some toroidal integer homology three-spheres are Seifert fibered, they are never hyperbolic.  The third author established that if all hyperbolic integer homology three-spheres have irreducible $SU(2)$-representations, then \Cref{conj:su2} holds in general.  While we are unable to complete the remaining step in this program, we confirm the existence of $SU(2)$-representations for toroidal integer homology three-spheres.  

%Recall that every prime summand of a closed three-manifold admits a JSJ torus decomposition \cite{JacoShalen, Johannson}, which for homology three-spheres turns out to be a minimal collection of incompressible tori such that the complementary pieces have either Seifert or hyperbolic geometries of finite volume.  (For the wary reader, Sol manifolds cannot have torus boundary and have non-trivial first homology.)  Hence, there are three (non-disjoint) types of prime homology three-spheres: Seifert fibered, hyperbolic, and toroidal (i.e., admitting an embedded $\pi_1$-injective torus).  We remark that although some toroidal homology three-spheres are Seifert fibered, i.e. if there are at least four singular fibers, they are never hyperbolic.  The third author established that if all hyperbolic homology three-spheres have non-trivial $SU(2)$- representations, then \Cref{conj:su2} must hold in general.  While we are unable to complete the remaining step in this program, we do confirm the existence of $SU(2)$-representations for toroidal homology three-spheres.  
 
\begin{theorem}\label{thm:toroidal}
Let $Y$ be a toroidal integer homology three-sphere.  Then $\pi_1(Y)$ admits an irreducible $SU(2)$-representation.  
\end{theorem}

A proof of Theorem~\ref{thm:toroidal} could be obtained by showing that toroidal integer homology three-spheres have non-trivial instanton Floer homology. Although we expect the latter to be true (see \cite[Problem 3.106]{Kirby}), we do not prove it in this article. Our proof of \Cref{thm:toroidal} instead relies on holonomy perturbations in a manner similar to the proof of \cite[Theorem 8.3]{Zentner}. If $Y$ is a toroidal integer homology three-sphere, then $Y$ can be viewed as a splice of knots $K_i$ in integer homology three-spheres $Y_i$ for $i = 1,2$ (see for example \cite[Proof of Corollary 6.2]{Eftekhary}). If some $Y_i$ has an irreducible $SU(2)$-representation, then there is a $\pi_1$-surjective map from $Y$ to $Y_i$ and we can pull back to an irreducible $SU(2)$-representation for $Y$.  If not, then we will study the image of the space of representations of the knot exterior $Y_i\setminus N(K_i)^\circ$ in the character variety for the boundary torus (i.e. in the pillowcase). Here, $N(K_i)$ denotes a closed tubular neighborhood of $K_i$, and $N(K_i)^\circ$ denotes its interior.  Similar to the case of non-trivial knots in $S^3$, if $Y_i$ has no irreducible representations, we will show that the image in the pillowcase contains a suitably essential loop. The loops for the two exteriors will have a non-trivial intersection, and therefore the spliced manifold $Y$ will admit an irreducible $SU(2)$-representation.  

Theorem~\ref{thm:toroidal} gives a simpler proof of \cite[Theorem 9.4]{Zentner} since it avoids the use of a finiteness result of Boileau-Rubinstein-Wang. 

\begin{corollary}[Theorem 9.4, \cite{Zentner}]
Every integer homology three-sphere other than $S^3$ has an irreducible $SL_2(\mathbb{C})$-representation of its fundamental group. 
\end{corollary}
\begin{proof} By the remarks above we have to consider three cases: Seifert fibered, hyperbolic, and toroidal integer homology three-spheres. 
Let $Y$ be an integer homology three-sphere other than $S^3$.  If $Y$ is hyperbolic, it admits an irreducible $SL_2(\mathbb{C})$-representation by lifting the holonomy representation to $PSL_2(\mathbb{C})$ \cite{Culler}.  If $Y$ is Seifert fibered, then $\pi_1(Y)$ admits an irreducible $SU(2)$-representation by \cite[Theorem 2.1]{Menagerie}.  If $Y$ is toroidal, the result now follows from Theorem~\ref{thm:toroidal}.  
\end{proof}

In order to generalize the holonomy perturbation machinery developed by the third author from non-trivial knots in $S^3$, we will need to establish a non-vanishing result which may be of independent interest.   

\begin{theorem}\label{instanton-0-surgery}
%Let $K$ be a knot in a homology three-sphere $Y$ such that the exterior of $K$  is irreducible and boundary-incompressible. If $I_*(Y) = 0$, then $I^w_*(Y_0(J)) \neq 0$.  
Let $J$ be a knot in an integer homology three-sphere $Y$ such that the exterior of $J$ is irreducible and boundary-incompressible.  Suppose that $I_*(Y) = 0$.  Then, $I^w_*(Y_0(J)) \neq 0$.  
\end{theorem}

Here, and throughout this article, $I_*$ denotes Floer's original version of instanton Floer homology and $I^w_*$ denotes instanton Floer homology for an admissible $SO(3)$-bundle with second Stiefel-Whitney class $w$.  (Note that $Y_0(J)$ admits only one such bundle.)  

The proof of \Cref{instanton-0-surgery} is a combination of (1) Kronheimer-Mrowka's non-vanishing result for instanton Floer homology of three-manifolds with a taut sutured manifold hierarchy \cite{KM_sutures}, (2) the surgery exact triangle in instanton Floer homology, and (3) Gordon's description of surgery on cable knots \cite{gordon}.  The argument is similar to Kronheimer-Mrowka's proof of {\em Property P} \cite{KM_P}.

While ~\Cref{instanton-0-surgery} itself may not be particularly interesting, it does lead to the following corollary, whose analogue in Heegaard Floer homology has been established by Ni \cite[p.1144]{Ni} and Conway and Tosun \cite{ConwayTosun}. The proof of the corollary appears in \Cref{sec:instantons} below. 

\begin{corollary}\label{cor:mazur}
Let $Y \neq S^3$ be an integer homology three-sphere which bounds a Mazur manifold.  Then, $I_*(Y) \neq 0$, and hence $\pi_1(Y)$ admits an irreducible $SU(2)$-representa\-tion.   
\end{corollary}

Recall that Baldwin-Sivek prove that if an integer homology three-sphere $Y$ bounds a Stein domain with non-trivial homology, then $\pi_1(Y)$ admits an irreducible $SU(2)$-representation \cite[Theorem 1.1]{BaldwinSivek}.  In light of \Cref{conj:su2}, the following conjecture would be a natural extension of their work:
\begin{conjecture}
If $Y \neq S^3$ is an integer homology three-sphere which bounds a Stein integer homology ball, then $\pi_1(Y)$ admits an irreducible $SU(2)$-representation.
\end{conjecture}
Since Stein domains admit handlebody decompositions with no three-handles \cite{Eliashberg}, Corollary~\ref{cor:mazur} proves this conjecture for the boundaries of Stein integer homology balls with the simplest possible handle decompositions.

Theorem~\ref{thm:toroidal} also has two simple corollaries.  The first one is obtained by considering branched covers over satellite knots in $S^3$. Remarkably, its proof requires no use of gauge theory, beyond our main result. Its proof appears in Section~\ref{sec:other} below.

\begin{corollary}\label{cor:branched-cover}
Let $K$ be a prime, satellite knot in $S^3$. \Cref{conj:su2} holds for any non-trivial cyclic branched cover of $K$.  
\end{corollary}
%
%\begin{proof}
%Suppose that $I_*(Z) = I^w_*(Z_0(J)) = 0$.  Then, by the surgery exact triangle, $I_*(Z_{1/n}(J)) = 0$ for all $n$.  In particular, $I_*(Z_{1/4}(J)) = 0$.  Note that $Z_{1/4}(J) = Z_{1}(C_{2,1}(J))$.  Therefore, by another iteration of the exact triangle, we see that $I^w_*(Z_0(C_{2,1}(J))) = 0$.  Following the notation of Gordon's Lemma 7.2, We can write $Z_0(C_{2,1}(J))$ as a gluing of the exterior of $J$ in $Z$ to 0-surgery on the cable knot $C_{2,1}$ in the solid torus.  The latter is irreducible and boundary-incompressible, and hence $Z_0(C_{2,1}(J))$ is obtained by gluing two irreducible, boundary-incompressible manifolds, and hence is an irreducible closed three-manifold with $b_1 = 1$.  This contradicts Proposition~\ref{prop:nonvanishing}, which says that $I^w_*(Z_0(C_{2,1}(J))) \neq 0$.  
%\end{proof}

To obtain the second corollary, define a graph manifold integer homology three-sphere to be a closed, orientable three-manifold whose torus decomposition has no hyperbolic pieces.\footnote{Some authors impose additional constraints, such as primeness or a non-trivial torus decomposition.}  As discussed above, the fundamental groups of Seifert integer homology three-spheres other than $S^3$ admit irreducible $SU(2)$-representations, and hence we obtain:

\begin{corollary}
Let $Y$ be a graph manifold integer homology three-sphere other than $S^3$.  Then $\pi_1(Y)$ admits an irreducible $SU(2)$-representation.
\end{corollary}

%\begin{remark}
A first alternate proof of this corollary can be obtained by noting that every integer homology three-sphere other than $S^3$ which is a graph manifold can be realized as the branched double cover of a non-trivial (arborescent) knot in $S^3$, see \cite{Bonahon-Siebenmann}.  A second alternate proof can be obtained by noting that every prime graph manifold integer homology three-sphere $Y$ other than $S^3$ or $\Sigma(2,3,5)$ admits a co-orientable taut foliation by \cite[Corollary 0.3]{BoileauBoyer}. This implies that $I_*(Y) \neq 0$, and this in turn implies that there exists an irreducible $SU(2)$-representation.  On the other hand, the binary dodecahedral group is well-known to admit two conjugacy classes of irreducible representations, completing the proof.  Note that, unlike for Seifert integer homology three-spheres, the Casson invariant of a non-trivial graph manifold can be zero. For example, the three-manifold $Y$ obtained as the splice of two copies of the exterior of the right handed trefoil has trivial Casson invariant \cite{FukuharaMaruyama, BoyerNicas}.  
%\end{remark}

\subsection*{Outline} In \Cref{sec:instantons} we establish the main technical result \Cref{instanton-0-surgery} whose strategy also leads us to prove \Cref{cor:mazur} about Mazur manifolds.  In \Cref{sec:pillowcase}, we review the pillowcase construction and prove \Cref{thm:toroidal} in subsection \ref{main}, using a technical result about invariance under holonomy perturbations in instanton Floer homology reviewed in \Cref{sec:perturbations}. The material in Section \ref{sec:perturbations} is mostly known (or at least folklore knowledge) and can be found elsewhere, but the reader might appreciate our synthesis of the role of holonomy perturbations and our sketch of invariance in order to follow more easily through the proof of our main results.  
%In Section~\ref{sec:perturbations}, we also relate the various types of perturbations needed for instanton Floer homology and for the third author's holonomy perturbations in the pillowcase.  
In \Cref{sec:other}, we prove \Cref{cor:branched-cover}.  
\subsection*{Acknowledgements} Tye Lidman was partially supported by NSF grant DMS-1709702 and a Sloan Fellowship. Juanita Pinz\'on-Caicedo is grateful to the Max Planck Institute for Mathematics in Bonn for its hospitality and financial support while a portion of this work was prepared for publication. She was partially supported by NSF grant DMS-1664567, and by Simons Foundation Collaboration grant 712377. Raphael Zentner is grateful to the DFG for support through the Heisenberg program.  We would also like to thank John Baldwin, Paul Kirk, and Tom Mrowka for helpful discussions.

\section{Instanton Floer homology of 0-surgery}\label{sec:instantons}
In this section we rely solely on formal properties of instanton Floer homology to prove \Cref{instanton-0-surgery} regarding the instanton Floer homology of $0$-surgeries, and \Cref{cor:mazur} regarding the instanton Floer homology of integer homology three-spheres that bound Mazur manifolds. More concrete aspects of instanton Floer homology groups, in particular those regarding perturbations, appear in \Cref{sec:perturbations} but in this section we wish to place the focus on the usefulness of formal properties for purposes of computations.

We consider instanton Floer homology for admissible bundles, as introduced by Floer \cite{floer}. For integer homology three-spheres, this is the trivial $SU(2)$-bundle over $Y$. For three-manifolds with positive first Betti number, this is an $SO(3)$-bundle $P \to Y$ such that there is a surface $\Sigma \subseteq Y$ on which the second Stiefel-Whitney class $w:=w_2(P)$ evaluates non-trivially, that is, such that $\langle w_2(P) , [\Sigma] \rangle \neq 0$.  The instanton Floer homology group is defined as a version of Morse homology of the Chern-Simons function on the space of connections on the admissible bundle \cite{floer, Donaldson}. It is denoted by $I_*(Y)$ for the trivial bundle on integer homology three-spheres, and it is denoted by $I^w_*(Y)$ for $SO(3)$-bundles $P \to Y$ with $w_2(P) = w$. We remark here that for an integer homology three-sphere, the trivial connection is isolated and is the unique reducible connection (up to gauge equivalence). In the other cases, the admissibility condition ensures that there are no reducible flat connections on the bundle.  %Very roughly, the critical points of the Chern-Simons function considered for instanton Floer homology yield irreducible representations of the fundamental group to $SO(3)$. \jpc{Isn't the whole point of section 4 to explain the way the critical points of $CS$ relate to representations? }

In the case of a knot $K$ in an integer homology three-sphere $Y$, there is a unique admissible bundle on the $0$-surgery $Y_0(K)$, because $H^2(Y_0(K);\Z/2) \cong \Z/2$.  Therefore, the instanton Floer homology group $I^w(Y_0(K))$ is defined without ambiguity. 
%\textcolor{red}
%{RZ: We need to say at least that Floer homology is defined for admissible bundles, which means either that $Y$ is a homology 3-sphere, or the condition on $w$, in particular, the notation has to be introduced at the beginning of the section.}

%$I^w_*$

\begin{proposition}\label{formal-prop} Instanton Floer homology satisfies the following properties:
\crefalias{enumi}{proposition}
\begin{enumerate}[label=(\arabic*), ref=\theproposition (\arabic*)]
\item\label{exact} For $Y$ an integer homology three-sphere and any $n\in\Z$, the three-manifolds $Y_{1/n}(K)$, $Y_{1/(n+1)}(K)$, and $Y_{0}(K)$ fit into an exact triangle $$\xymatrix{&I_*(Y_{1/n}(K))\ar[dr]&\\I_*(Y_{1/(n+1)}(K))\ar[ru]&&I_*^w(Y_{0}(K)).\ar[ll]}$$ 
\item\label{prop:nonvanishing} If $M$ is an irreducible three-manifold with $b_1(M) = 1$, then $I^w_*(M) \neq 0$.  
\item\label{homologysu2cyclic} For $Y$ an integer homology three-sphere, if $\pi_1(Y)$ admits no irreducible $SU(2)$-representations, then $I_*(Y) = 0$.  
\item\label{s2xs1} $I^w_*(S^2\times S^1)=0$.
\end{enumerate}
\end{proposition}

\begin{proof} The surgery exact triangle in \eqref{exact} is originally due to Floer \cite[Theorem 2.4]{floer} with details given in \cite[Theorem 2.5]{Braam-Donaldson}. The non-triviality result in \eqref{prop:nonvanishing} is precisely \cite[Theorem 7.21]{km-sutures}. Next, \eqref{homologysu2cyclic} follows from \cite[Theorem 1]{FloerOriginal}, since if $\pi_1(Y)$ admits no irreducible $SU(2)$-representations, then the generating set for the instanton Floer chain groups is empty. Finally, \eqref{s2xs1} follows from \eqref{homologysu2cyclic} and \eqref{exact}, by considering the surgery exact triangle for surgery on the unknot in $S^3$.  Alternatively, this follows from the definition of $I^w_*$ (see \Cref{sec:perturbations}), since $\pi_1(S^2 \times S^1)$ admits no representations to $SO(3)$ which do not lift to $SU(2)$.  %can be proved by exactly the same argument.  If $H_1(M)$ is infinite cyclic and there are no irreducible $SO(3)$ representations of $\pi_1(M)$, then $I^w_*(M) = 0$.  (A weaker condition also holds.  In fact, $I^w_*(M) = 0$ if all representations of $\pi_1(M)$ to $SO(3)$ lift to $SU(2)$.)  
%The theory of sutured manifolds and their decompositions [Ga1] provides a (finite depth) taut foliation on every closed orientable irreducible 3-manifold with non- trivial integral second homology.
%results of Eliashberg- Thurston [ET] and Vogel [Vo1] give a correspondence between the two worlds, namely an essentially unique way to deform a taut foliation without a torus leaf into a tight contact structure.
\end{proof}

We will be particularly interested in integer homology three-spheres whose fundamental groups do not admit irreducible $SU(2)$-representations.  We therefore establish the following definition.
\begin{definition}
An integer homology three-sphere $Y$ is {\em $SU(2)$-cyclic} if every $SU(2)$-representation of $\pi_1(Y)$ is trivial.  
\end{definition} 
Notice that \Cref{conj:su2} states that $S^3$ is the only $SU(2)$-cyclic integer homology three-sphere. %\textcolor{red}{RZ: The original sentence here sounded a bit odd to me. I have commented it out and replaced it.}
%Of course, by \Cref{conj:su2}, $S^3$ is the only $SU(2)$-cyclic integer homology three-sphere. 

Having stated the above formal properties of instanton Floer homology, the proofs of \Cref{instanton-0-surgery} and \Cref{cor:mazur} now follow easily.

%\begin{reptheorem}{instanton-0-surgery} Let $Y$ be an integer homology three-sphere such that $I_*(Y)=0$ and let $K\subset Y$ be a knot whose exterior is irreducible and boundary-incompressible. Then, $I^w_*(Y_0(K)) \neq 0$.  
%\end{reptheorem}

\subsection{Non-vanishing of Instanton Floer Homology}
%1.3: non-vanishing of 0-surgery along knot with exterior irreducible and boundary-incompressible.
%1.4: non-vanishing for boundary of Mazur

In this subsection we illustrate the way the formal properties from \Cref{formal-prop} can be used to show that the instanton homology groups are non-zero in two cases: (1) three-manifolds obtained as 0-surgery along knots in $SU(2)$-cyclic integer homology three-spheres whose exterior is irreducible and boundary incompressible, and (2) three-manifolds other than $S^3$ obtained as the boundary of a Mazur manifold.

\begin{proof}[Proof of \Cref{instanton-0-surgery}] 
We assume $I^w_*(Y_0(K))$ is trivial and argue by contradiction.  By \Cref{exact} the three-manifolds $Y_{1/n}(K)$, $Y_{1/(n+1)}(K)$, and $Y_{0}(K)$ fit together in an exact triangle $$\xymatrix{&I_*(Y_{1/n}(K))\ar[dr]&\\I_*(Y_{1/(n+1)}(K))\ar[ru]&&I_*^w(Y_{0}(K)).\ar[ll]}$$ 
The assumption $I^w_*(Y_0(K)) = 0$ implies that there is an isomorphism $$I_*(Y_{1/(n+1)}(K)) \cong I_*(Y_{1/n}(K)) \text{ for each } n\in\Z.$$ In particular, if $n=0$ then $I_*(Y_{1}(K))\cong I_*(Y)=0$ thus showing that for all $n\in\Z$, 
\begin{equation}\label{1/n}I_*(Y_{1/n}(K))=0. \end{equation} 
Now, %consider a closed curve $C_{2,1}\subset S^1\times \partial D_{\scaleto{1/2}{5pt}}^2\subset S^1\times D^2 $ representing the class $2\cdot[S^1]+[\partial D_{\scaleto{1/2}{5pt}}^2]$ in $H_1(S^1\times \partial D_{\scaleto{1/2}{5pt}}^2)$. Then, if $h:S^1\times D^2\to N(K)$ is the diffeomorphism that knots a standard solid torus into a tubular neighborhood for $K$, the $(2,1)$-cable of $K$ is the image of $C_{2,1}$ under $h$. A 
a result of Gordon \cite[Lemma 7.2]{gordon} shows that $Y_{1/4}(K)$ is diffeomorphic to $Y_{1}(K_{2,1})$, where $K_{2,1}$ is the $(2,1)$-cable of the knot $K$ (See \Cref{fig:satellite} for an example of $K_{2,1}$). This together with \Cref{1/n} implies $I_*(Y_{1}(K_{2,1}))=0$. An iteration of an exact triangle as in \Cref{exact} for surgeries along $K_{2,1}$ gives $I^w_*(Y_{0}(K_{2,1}))= 0$. 

We now consider a decomposition of $Y_{0}(K_{2,1})$ that includes the knot exterior of $K$ in $Y$.  Denote by $C_{2,1}$ a closed curve that lies in the boundary of a  ``small'' solid torus $ S^1\times \partial D_{\scaleto{1/2}{5pt}}^2\subset S^1\times D^2 $, and representing the class $2[S^1]+[\partial D_{\scaleto{1/2}{5pt}}^2]$ in $H_1(S^1\times \partial D_{\scaleto{1/2}{5pt}}^2)$. Notice that the $0$-framing of $K_{2,1}$ in $Y$ induces the framing on $C_{2,1}$ determined by the curve $\lambda$ in $\partial N\left(C_{2,1}\right)$ that represents the class $2[S^1]$ in $H_1(S^1\times \partial D^2)$ (see \cite[pg. 692]{gordon}). Therefore, the manifold $Y_{0}(K_{2,1})$ can be expressed as the union of the knot exterior $Y\setminus N(K)$, and the result of Dehn surgery on $S^1\times D^2$ along the curve $C_{2,1}$ with framing given by $\lambda$. By hypothesis the knot exterior $Y\setminus N(K)$ is irreducible and boundary-incompressible, and by \cite[Lemma 7.2]{gordon} the 0-surgery along the curve $C_{2,1}$ is a Seifert fibred space with incompressible boundary. Hence $Y_{0}(K_{2,1})$ is an irreducible closed three-manifold with first Betti number equal to $1$ and with trivial instanton Floer homology. However, this contradicts \Cref{prop:nonvanishing}, which says that $I^w_*(Y_0(K_{2,1})) \neq 0$.  
\end{proof}

%All of the previous formal properties of instanton homology allows one to prove \Cref{cor:mazur} 
Next, consider integer homology three-spheres that bound a Mazur manifold, that is, a four-manifold that admits a handle decomposition in terms of exactly one 0-handle, one 1-handle, and one 2-handle as in \Cref{Mazur}. Then we have:

\begin{proof}[Proof of \Cref{cor:mazur}]
If $Y$ bounds a Mazur manifold, then there exists a knot $J$ in $Y$ such that $Y_0(J) = S^2 \times S^1$.  Moreover, if $I_*(Y) = 0$, a combination of the surgery exact triangle from ~\Cref{exact} and the computation $I^w_*(S^2 \times S^1) = 0$ from~\Cref{s2xs1} shows once again that $I_*(Y_{1/4}(J)) = 0$. The same argument used above in the proof of \Cref{instanton-0-surgery} then gives $I_*(Y_0(J_{2,1})) = 0$.  However, it is easy to see that the exterior of a knot in $S^2 \times S^1$ which generates homology is either irreducible and boundary-incompressible or a solid torus.  The latter case corresponds to $Y = S^3$, so by assumption we have that $Y_0(J_{2,1})$ is irreducible with $b_1 = 1$.  But this contradicts \Cref{prop:nonvanishing}.     
\end{proof}

\begin{figure}
    \centering
    \def\svgwidth{0.3\textwidth}
    %% Creator: Inkscape 1.0beta1 (32d4812, 2019-09-19), www.inkscape.org
%% Accompanies image file 'Mazur.pdf' (pdf, eps, ps)
\begingroup%
  \makeatletter%
  \providecommand\color[2][]{%
    \errmessage{(Inkscape) Color is used for the text in Inkscape, but the package 'color.sty' is not loaded}%
    \renewcommand\color[2][]{}%
  }%
  \providecommand\transparent[1]{%
    \errmessage{(Inkscape) Transparency is used (non-zero) for the text in Inkscape, but the package 'transparent.sty' is not loaded}%
    \renewcommand\transparent[1]{}%
  }%
  \providecommand\rotatebox[2]{#2}%
  \newcommand*\fsize{\dimexpr\f@size pt\relax}%
  \newcommand*\lineheight[1]{\fontsize{\fsize}{#1\fsize}\selectfont}%
  \ifx\svgwidth\undefined%
    \setlength{\unitlength}{240.30021429bp}%
    \ifx\svgscale\undefined%
      \relax%
    \else%
      \setlength{\unitlength}{\unitlength * \real{\svgscale}}%
    \fi%
  \else%
    \setlength{\unitlength}{\svgwidth}%
  \fi%
  \global\let\svgwidth\undefined%
  \global\let\svgscale\undefined%
  \makeatother%
  \begin{picture}(1,1.28993674)%
    \lineheight{1}%
    \setlength\tabcolsep{0pt}%
    \put(0,0){\includegraphics[width=\unitlength,page=1]{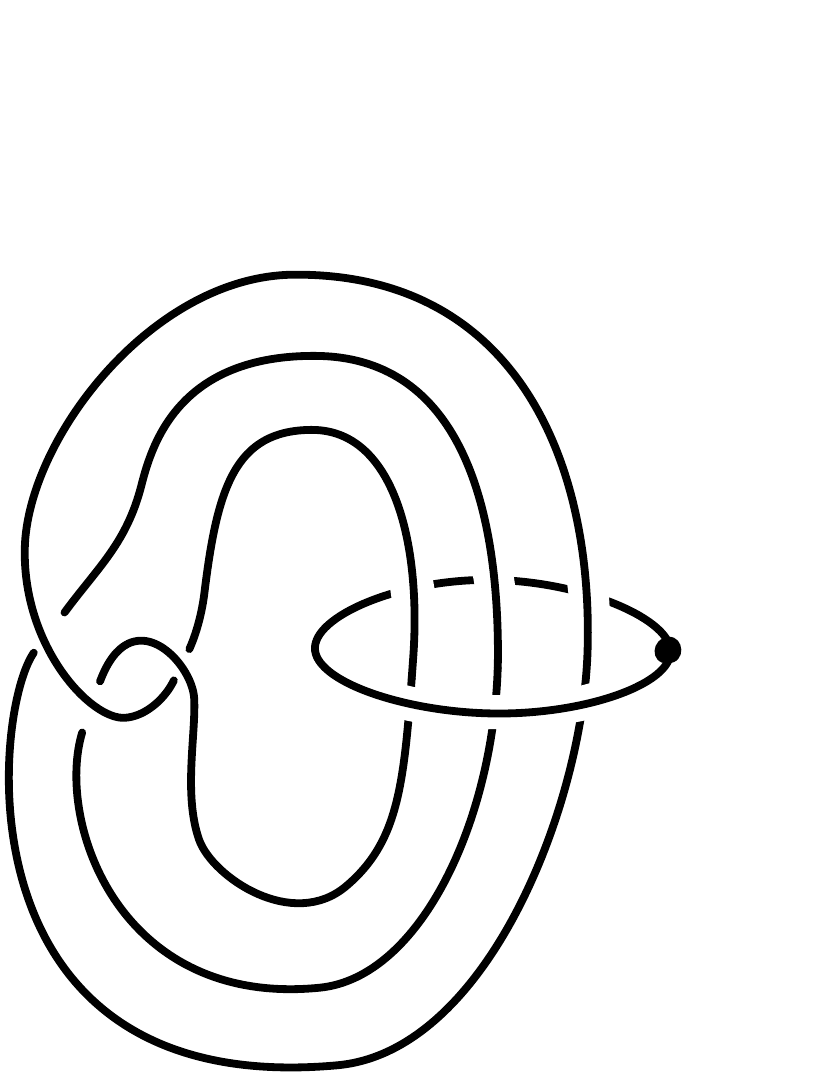}}%
    \put(0.30648546,0.99582491){\color[rgb]{0,0,0}\makebox(0,0)[lt]{\lineheight{1.25}\smash{\begin{tabular}[t]{l}$n$\end{tabular}}}}%
  \end{picture}%
\endgroup%

    \caption{A Mazur manifold with one two-handle attached with framing given by $n$ for some $n\in\N$.}\label{Mazur}
\end{figure}

%Suppose that $Z$ is an arbitrary homology three-sphere with no $SU(2)$-representations, and $J$ is any non-trivial knot in $Z$ whose exterior is irreducible.
%
%%Then: 
%Denote by $Z_f(J)$ the result of performing surgery along $J$ with framing number $f$. Then the previous assumptions imply that both $Z_{+1}(J)$ and $Z_{-1}(J)$ have non-trivial instanton Floer homology and thus $SU(2)$-representations.
%
%In addition, infinitely many other $1/n$-surgeries on $J$ must also have non-trivial instanton Floer homology and therefore $SU(2)$-representations.  %(In fact, for no value of n can 1/n and 1/(n+1) surgery both have trivial instanton Floer homology.  We can also get some other slopes as well.)  
%
%the instanton Floer homology of 0-surgery on $J$ is non-trivial:
%If $Z_0(J)$ is irreducible, then it has a taut foliation and you should then be able to put it in a symplectic four-manifold in the usual way.  
%If $Z_0(J)$ is not irreducible, it is of the form $N \# M$, where $N$ is an irreducible three-manifold such that $H_1(N;\Z) = \Z$ and $M$ is a homology three-sphere which may or may not have $SU(2)$-representations.  An irreducible $SU(2)$-representation for $Z_0(J)$ may be obtained by pinching to $N$, and $N$ itself also has a taut foliation and hence embeds in a suitable symplectic four-manifold.   

\section{The pillowcase alternative}\label{sec:pillowcase}
In this section we recall the relevant background on $SU(2)$-character varieties and generalize work of the third author \cite{Zentner} to prove \Cref{thm:toroidal}, our main result. %\jpc{We should maybe say what we have in mind... tie it with the rest of our story.}

\subsection{The pillowcase}
Given a connected manifold $M$, we denote by $$R(M) = \Hom(\pi_1(M),SU(2))/SU(2)$$ the space of $SU(2)$-representations of its fundamental group, up to conjugation. We will write $R(M)^*$ for the subset of irreducible representations.  For example, the space $R(T^2)$ 
is identified with the {\em pillowcase}, an orbifold homeomorphic to a two-dimensional sphere with four corner points. To see this, notice that since $\pi_1(T^2) \cong \Z^2$ is abelian, the image of any representation $\rho\colon \pi_1(T^2) \to SU(2)$ is contained in a maximal torus subgroup of $SU(2)$. Up to conjugation, this torus can be identified with the circle group consisting of matrices of the form $\begin{bmatrix} e^{i \theta} & 0 \\ 0 & e^{-i \theta} \end{bmatrix}$ for $\theta \in [0,2\pi]$. Thus, if we denote the generators of $\pi_1(T^2) \cong \Z^2$ by $m$ and $l$, then, again after conjugation, a representation $\rho\in R(T^2)$ is determined by
\[
	\rho(m) = \begin{bmatrix} e^{i \alpha} & 0 \\ 0 & e^{-i \alpha} \end{bmatrix} \hspace{0.5cm} 
	\text{and} \hspace{0.5cm}
	\rho(l) = \begin{bmatrix} e^{i \beta} & 0 \\ 0 & e^{-i \beta} \end{bmatrix},
\]
and hence we can associate to $\rho$ a pair $(\alpha,\beta) \in [0,2\pi] \times [0,2 \pi]$. However, conjugation of $\rho$ by the element $\begin{bmatrix} 0 & 1 \\ -1 & 0 \end{bmatrix}$ gives rise to the representation associated to the pair $(2\pi - \alpha, 2 \pi - \beta)$. This is the only ambiguity, however, as can be seen using the fact that the trace of an element in $SU(2)$ determines its conjugacy class. Therefore $R(T^2)$ is isomorphic to %the quotient of the torus $T$ by the hyperelliptic involution $\tau\colon (\alpha,\beta) \mapsto (-\alpha,-\beta)$. This has four fixed points, and its quotient
%\begin{equation*} R(T^2) = T/\tau \end{equation*} 
%is homeomorphic to a two-dimensional sphere. It can also be seen as 
the quotient of the fundamental domain $[0,\pi] \times [0,2 \pi]$ by identifications on the boundary as indicated in \Cref{pillowcase trefoil}.

\begin{figure}[h!]
\def\svgwidth{0.75\textwidth}\input{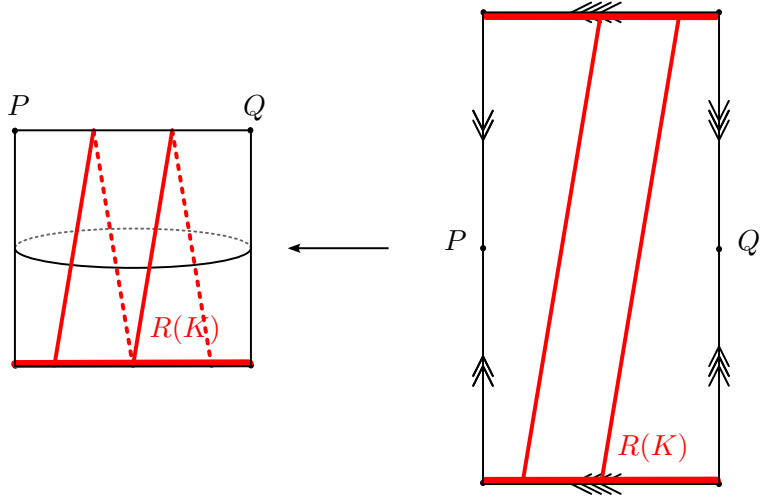}
\caption{The gluing pattern for obtaining the pillowcase from a rectangle, 
and the image of the representation variety $R(K)$ of the trefoil in the pillowcase.}
\label{pillowcase trefoil}
\end{figure}

If we have a three-manifold $M$ with torus boundary, then the inclusion $i\colon T^2 \cong \partial M \hookrightarrow M$ induces a map $i^*\colon R(M) \to R(T^2)$ by restricting a representation to the boundary. For instance, if $K$ is a knot in a three-manifold $Y$, then the three-manifold $Y(K) := Y \setminus N(K)^\circ$ obtained by removing the interior of a tubular neighborhood $N(K)$ of $K$ from $Y$, is a three-manifold with boundary a two-dimensional torus. \Cref{pillowcase trefoil} shows the image of $R(S^3(K))$ when $K$ is the right handed trefoil in $S^3$, once in the pillowcase, and once in the fundamental domain $[0,\pi] \times [0,2 \pi]$. Here we use the convention that the first coordinate corresponds to $\rho(m_K)$, where $m_K$ is a meridian to the knot $K$, and the second coordinate corresponds to $\rho(l_K)$, where $l_K$ is a longitude of the knot $K$. 

For a knot $K$ in a three-manifold $Y$ there is a well-defined notion of meridian $m_K$, and if the knot is nullhomologous, there is a well-defined notion of longitude $l_K$. In particular, this is the case for any knot $K$ in an integer homology three-sphere $Y$. In what follows, we will use the notation $R(K):= R(Y(K))$ if it is clear which integer homology three-sphere $Y$ we have in mind, and we will stick to the above convention of the coordinates in $R(T^2)$ corresponding to the meridian and longitude of $K$. With these conventions, all abelian representations in $R(K)$ map under $i^*$ to the thick red line $\{ \beta = 0 \mod{2\pi \Z} \}$ `at the bottom' of the pillowcase $R(T^2)$. Indeed, $l_K$ is a product of commutators in the fundamental group of the knot complement, so an abelian representation necessarily maps $l_K$ to the identity.  Furthermore, for any $\alpha\in [0,\pi]$ we can find an abelian representation of $R(K)$ whose restriction to $R(T^2)$ corresponds to $(\alpha,0)$.

If we cut the pillowcase open along the lines $a_0:= \{\alpha = 0 \mod 2 \pi \mathbb{Z}\}$ and $a_\pi:=\{\alpha = \pi \mod 2 \pi \mathbb{Z}\}$, we obtain a cylinder $C = [0,\pi] \times \R/2 \pi \Z$.  In the gluing pattern of Figure \ref{pillowcase trefoil} this means that we do not perform the identifications along the four indicated vertical boundary lines.  

Our main goal is to prove Theorem~\ref{pillowcase alternative} below, which asserts that if $K$ is a knot in an $SU(2)$-cyclic integer homology three-sphere whose 0-surgery has non-trivial instanton homology, then 
the image of $R(K)$ in the pillowcase contains a homologically non-trivial embedded closed curve in the cylinder $C$. In order to derive Theorem \ref{thm:toroidal} from this, we need a more refined statement, namely, that there is a homologically non-trivial embedded closed curve in $i^*(R(K))$ that is disjoint from a neighborhood of the two lines $a_0$ and $a_\pi$. Notice that for a knot in $S^3$, there are no representations with $\rho(l_K) \neq \id$ and $\rho(m_K) = \pm \id$. This is because the fundamental group of a knot complement in $S^3$ is normally generated by the meridian of the knot. In particular, there are no representations in $i^*(R(K))$ that have coordinates $(\alpha,\beta)$ with $\beta \neq 0$, and $\alpha = 0$ or $\alpha = \pi$.  In \cite[Proposition 8.1]{Zentner}, it is shown that the image of $R(K)^*$, the subset of irreducible representations in $R(K)$, in fact stays outside a neighborhood of these two lines. We begin with a generalization of this fact. 

\begin{lemma}\label{lem:avoids-lines}
Let $K$ be a knot in an $SU(2)$-cyclic integer homology three-sphere $Y$.  There is a neighborhood of the lines $\{\alpha = 0 \mod 2 \pi \mathbb{Z}\}$ and $\{\alpha = \pi \mod 2 \pi \mathbb{Z}\}$ in the pillowcase which is disjoint from the image of $R(K)^*$.  
\end{lemma}
\begin{proof} Suppose by contradiction that the image of $R(K)^*$ intersects every neighborhood of the lines $\{\alpha = 0 \mod 2 \pi \mathbb{Z}\}$ and $\{\alpha = \pi \mod 2 \pi \mathbb{Z}\}$. If that was the case, then we could find a sequence of elements in $R(K)^*$ whose image under $i^*$ converges to a point on one of the two lines.  By the compactness of $R(K)$, the limit is the image of a representation $\rho \colon \pi_1(Y(K)) \to SU(2)$ sending every meridional curve $\mu$ to $\pm 1$.  We first claim that $\rho$ must be a central representation (and hence reducible), and so its image under $i^*$ can only be $(0,0)$ or $(\pi,0)$.  

First, if $\rho(\mu) = 1$, then $\rho\colon \pi_1(Y(K)) \to SU(2)$ is really a representation of $\pi_1(Y)$. Since $Y$ is assumed to be $SU(2)$-cyclic, then the representation is trivial and therefore $\rho(\lambda) = 1$. As a consequence, if the limit of elements in $R(K)^*$ is an element of the line $\{ \alpha = 0 \mod 2 \pi \mathbb{Z} \}$, then it is the point $(0,0)$ in the pillowcase.  Next, consider the case that $\rho(\mu) = -1$. If the representation $\rho$ is irreducible, then we obtain an irreducible representation $\widetilde{\rho}\colon \pi_1(Y)  \to SO(3)$. The obstruction to lifting an $SO(3)$ representation into an $SU(2)$-representation is an element of $H^2(Y;\Z/2)$, and since $Y$ is an integer homology three-sphere, the obstruction vanishes and $\widetilde{\rho}$ would lift to an irreducible representation to $SU(2)$, contradicting the fact that $Y$ is $SU(2)$-cyclic.  Therefore, a representation $\rho \colon \pi_1(Y(K)) \to SU(2)$ satisfying $\rho(\mu) = -1$ is reducible and hence abelian, and so factors through $H_1(Y(K))$.  Because $\lambda$ is trivial in $H_1(Y(K))$, we see that $\rho$ is the central representation sending $\mu$ to $-1$ and $\lambda$ to $1$, and this corresponds to the point $(-\pi,0)$ in the pillowcase.  All of this shows that if a sequence of elements in $i^*R(K)$ converges to a point on the lines $\{\alpha = 0 \mod 2 \pi \mathbb{Z}\}$ and $\{\alpha = \pi \mod 2 \pi \mathbb{Z}\}$, then the limit point is a central representation. For notation, we will call these representations $\rho_\pm$ for the sign of the image of $\mu$.  

Now, it remains to show that the points $(0,0)$ and $(\pi,0)$ cannot be limits of irreducible representations. We remark here that this fact does not require that $Y$ is $SU(2)$-cyclic. Let $\Gamma=\pi_1\left(Y(K)\right)$. A result of Weil \cite{Weil} expanded in \cite[Chapter 2]{lubotzky-magid} shows that $T_\rho R(K)$ corresponds to $H^1(\Gamma;\mathfrak{su}(2)_{\text{ad} \smallcirc \rho})$. This group is identified with the first cohomology group (with twisted coefficients) of a $K(\Gamma,1)$-space, or more generally, with the first (twisted) cohomology of any CW complex with fundamental group isomorphic to $\Gamma$. This shows that $H^1(\Gamma;\mathfrak{su}(2)_{{\text{ad} \smallcirc \rho}})=H^1(Y(K);\mathfrak{su}(2)_{\text{ad} \smallcirc \rho})$ and so $T_\rho R(K)=H^1(Y(K);\mathfrak{su}(2)_{\text{ad} \smallcirc \rho})$. Next, since each representation $\rho_\pm$ is central, then $ad\circ\rho_\pm$ is the trivial representation and so $$H^1\left(Y(K);\mathfrak{su}(2)_{\text{ad} \smallcirc \rho_\pm}\right)=H^1\left(Y(K);\R^3\right)\cong \R^3.$$ This shows that the tangent space to $R(K)$ at $\rho_\pm$ is three-dimensional. Since we obtain three dimensions of freedom by abelian representations near $\rho_\pm$ in $R(K)$, the entire tangent space to $R(K)$ consists of tangent vectors to abelian representations and so there cannot be irreducible representations near $\rho_\pm$, completing the proof. 
\end{proof}

\subsection{Essential curves in the pillowcase}
In this section, we relate the instanton Floer homology of 0-surgery on a knot to the image of the character variety of the knot exterior in the pillowcase.  This will be the key step in the proof of Theorem~\ref{thm:toroidal}, found at the end of this subsection.  
  
We next establish some notation, following Kronheimer-Mrowka in \cite{KM_Dehn}, that will be useful in the proof of our next theorem.  
\begin{definition}\label{def:subset pillowcase}
For a subset $L \subseteq R(T^2)$, we denote by $R(K |  L)$ the set of elements $[\rho] \in R(K)$ such that $[i^*\rho] \in L$. 
\end{definition}

\begin{theorem}\label{thm:pillowcase intersection}
Let $K$ be a knot in an integer homology three-sphere $Y$, and assume that the instanton Floer homology of the $0$-surgery is non-vanishing, $I^w_*(Y_0(K)) \neq 0$. Then any topologically embedded path from $P=(0,\pi)$ to $Q=(\pi,\pi)$ in the associated pillowcase has an intersection point with the image of $R(K)$. 
\end{theorem}

Before proving the theorem, we point out that this generalizes \cite[Theorem 7.1]{Zentner}, from knots in $S^3$ to knots in general integer homology three-spheres. The main difference in the argument compared to \cite[Theorem 7.1]{Zentner} is that here we make use of the non-trivial instanton Floer homology of the 0-surgery in an essential way, which is exploited through its connection with holonomy perturbations of the Chern-Simons functional.  The arguments of the third author in \cite{Zentner} instead use holonomy perturbations of a moduli space which computes the Donaldson invariants of a closed 4-manifold containing the 0-surgery as a hypersurface. In that case, the non-vanishing result builds on the existence of a taut foliation on $S^3_0(K)$ for a non-trivial knot $K$. In the case at hand, we do not know whether $Y_0(K)$, the $0$-surgery on a knot $K$ in the integer homology three-sphere $Y$, admits a taut foliation.  

\begin{proof}
%Let $\gamma$ be such a path.  Following \cite[]{Zentner}, there exists a holonomy perturbation $\pi$ for the admissible bundle over $Y_0(K)$ such that 
%\begin{enumerate}
%\item solutions to the $\pi$-perturbed flat equations corresponds to intersection points between $i^*(R(K))$ and a sufficiently small perturbation of $\gamma$;
%\item we have non-degeneracy and irreducibility of the critical points;  
%\item we have transversality for all the moduli spaces.
%\end{enumerate}
%The result now follows from Proposition~\ref{prop:perturbed-floer-homology}. 

Suppose by contradiction that there is a continuous embedded path $c$ from $P$ to $Q$ such that its image is disjoint from $i^*(R(K)) \subseteq R(T^2)$. (We will not distinguish between paths and their image for the remainder of this proof.)  In other words, $R(K |  c)$ is empty.  In particular, we may suppose that $c$ is disjoint from the bottom line $\{\beta = 0\}$ of the pillowcase $R(T^2)$, since any element of this line lies in the image of $i^*$. Since the image $i^*(R(K))$ is compact, there is a neighborhood $U \subseteq R(T^2)$ of the image of $c$ in $R(T^2)$ which is still disjoint from $i^*(R(K))$.  Since $R(K |  c)$ is empty, for $c'$ sufficiently close to $c$, $R(K |  c')$ is empty as well.  

Associated to a three-manifold and admissible bundle, we consider two objects: the Chern-Simons functional and holonomy perturbations of the Chern-Simons functional. These are described in detail in \Cref{sec:perturbations}, in particular Sections \ref{sec:hol perturbations} and \ref{sec:holonomy perturbations and shearing maps}, but their definition is not needed for the proof.  Given a three-manifold $Z$ with admissible bundle represented by $w$ and a holonomy perturbation $\Psi$, let $R^w_\Psi(Z)$ denote the set of critical points of the Chern-Simons functional perturbed by $\Psi$.  By Theorem \ref{thm: hol perturbations, missing curve} below (which is essentially a synthesis of \cite[Theorem 4.2 and Proposition 5.3]{Zentner}), there exists a path $c'$ arbitrarily close to $c$ and a (holonomy) perturbation $\Psi$ of the Chern-Simons functional such that $R^w_\Psi(Y_0(K))$ is a double cover of $R(K |  c')$. Therefore, $R^w_\Psi(Y_0(K))$ is empty, so computing Morse homology with respect to this perturbation of the Chern-Simons functional produces a trivial group.  However, Theorem \ref{thm:invariance of instanton Floer homology} below asserts that computing Morse homology with respect to the particular perturbation $\Psi$ produces a group isomorphic to $I^w_*(Y_0(K))$, which is non-zero by assumption. Therefore, we obtain a contradiction.  
\end{proof}

\begin{remark} Although \cite[Proposition 5.3]{Zentner} is only stated for knots in $S^3$, the arguments used in its proof apply for a knot in an arbitrary $SU(2)$-cyclic integer homology three-sphere. 
\end{remark}

If we combine the constraint that $Y$ is $SU(2)$-cyclic with the assumption that $I^w_*(Y_0(K))$ is non-trivial, then we obtain the following generalization of \cite[Theorem 7.1]{Zentner}, which will be the last step before the proof of our main theorem.

\begin{theorem}(Pillowcase alternative) \label{pillowcase alternative}
	Suppose $Y$ is an $SU(2)$-cyclic integer homology three-sphere. Suppose $K$ is a knot in $Y$ such that the 0-surgery $Y_0(K)$ has non-trivial instanton Floer homology $I^w_*(Y_0(K))$, where $w$ is the non-zero class in $H^2(Y_0(K);\Z/2) \cong \Z/2$. Then the image $i^*(R(Y(K)))$ in the cut-open pillowcase $C = [0,\pi] \times (\R/2 \pi \Z)$ contains a topologically embedded curve which is homologically non-trivial in $H_1(C;\Z) \cong \Z$.
\end{theorem}
\begin{proof}
	The hypothesis implies that the lines $\{ (0,\beta) \in R(T^2) \, | \, \beta \neq 0 \}$ and $\{ (\pi,\beta) \in R(T^2) \, | \, \beta \neq 0 \}$ have empty intersection with $i^*(R(Y(K)))$. The conclusion then follows from Theorem \ref{thm:pillowcase intersection} together with the Alexander duality argument of \cite[Lemma 7.3]{Zentner}. 
\end{proof}

\begin{figure}[h!]
\def\svgwidth{0.3\textwidth}
%% Creator: Inkscape 1.0beta1 (32d4812, 2019-09-19), www.inkscape.org
%% Accompanies image file 'pillowcase_sad_case.pdf' (pdf, eps, ps)
\begingroup%
  \makeatletter%
  \providecommand\color[2][]{%
    \errmessage{(Inkscape) Color is used for the text in Inkscape, but the package 'color.sty' is not loaded}%
    \renewcommand\color[2][]{}%
  }%
  \providecommand\transparent[1]{%
    \errmessage{(Inkscape) Transparency is used (non-zero) for the text in Inkscape, but the package 'transparent.sty' is not loaded}%
    \renewcommand\transparent[1]{}%
  }%
  \providecommand\rotatebox[2]{#2}%
  \newcommand*\fsize{\dimexpr\f@size pt\relax}%
  \newcommand*\lineheight[1]{\fontsize{\fsize}{#1\fsize}\selectfont}%
  \ifx\svgwidth\undefined%
    \setlength{\unitlength}{91.87500572bp}%
    \ifx\svgscale\undefined%
      \relax%
    \else%
      \setlength{\unitlength}{\unitlength * \real{\svgscale}}%
    \fi%
  \else%
    \setlength{\unitlength}{\svgwidth}%
  \fi%
  \global\let\svgwidth\undefined%
  \global\let\svgscale\undefined%
  \makeatother%
  \begin{picture}(1,1.1176087)%
    \lineheight{1}%
    \setlength\tabcolsep{0pt}%
    \put(0,0){\includegraphics[width=\unitlength,page=1]{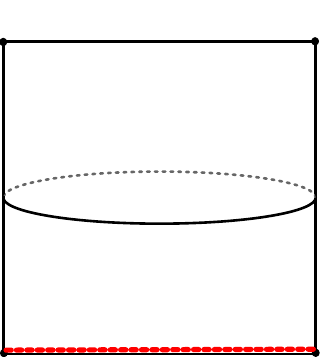}}%
    \put(0.49138201,0.10144818){\color[rgb]{1,0,0}\makebox(0,0)[lt]{\lineheight{0}\smash{\begin{tabular}[t]{l}$i^*R(K)$\end{tabular}}}}%
    \put(0,0){\includegraphics[width=\unitlength,page=2]{pillowcasesadcase.pdf}}%
    \put(1.00306116,1.0170985){\color[rgb]{0,0,0}\makebox(0,0)[lt]{\lineheight{1.25}\smash{\begin{tabular}[t]{l}$Q$\end{tabular}}}}%
    \put(-0.00376276,1.01505768){\color[rgb]{0,0,0}\makebox(0,0)[lt]{\lineheight{1.25}\smash{\begin{tabular}[t]{l}$P$\end{tabular}}}}%
  \end{picture}%
\endgroup%

\caption{This is a hypothetical image of a representation variety $i^*(R(K))$ of a knot $K$ in an integer homology three-sphere $Y$. The homology three-sphere $Y$ is assumed to satisfy $I^w_*(Y_0(K))\neq 0$ and assumed to not be $SU(2)$-cyclic. As a consequence, $i^*(R(K))$ intersects every path joining $P$ and $Q$ as in \Cref{thm:pillowcase intersection}, but it does not contain a curve which is homologically non-trivial in the cut-open pillowcase $C = [0,\pi] \times (\R/2 \pi \Z)$. This hypothetical example thus illustrates that the $SU(2)$-cyclic assumption is necessary in \Cref{pillowcase alternative}.}
\label{figure:sad pillowcase}
\end{figure}

\subsection{Main Result}\label{main}
In this subsection we prove that if an integer homology three-sphere contains an embedded incompressible torus, then the fundamental group of the homology three-sphere admits irreducible $SU(2)$-representations. To derive our result we first recall that we can realize a toroidal integer homology three-sphere as a splice, as in \cite[Proof of Corollary 6.2]{Eftekhary}.  We then study the image of the two knot exteriors in the pillowcase of the incompressible torus. With this in mind, we include the following definition.

\begin{definition} 
Let $K_1\subset Y_1$ and $K_2\subset Y_2$ be oriented knots in oriented integer homology three-spheres. For $i=1,2$, denote by $\mu_i,\lambda_i\subset \partial N(K_i)$ a meridian and longitude for $K_i$ in $Y_i$. Form a three-manifold $Y$ as $$\left(Y_1\setminus N(K_1)^\circ\right)\underset{h}{\cup}\left(Y_2\setminus N(K_2)^\circ\right),$$ where $h\colon \partial N(K_1)\to \partial N(K_2)$ identifies $\mu_1$ with $\lambda_2$, and $\lambda_1$ with $\mu_2$. The manifold $Y$ is called the {\em splice} of $Y_1$ and $Y_2$ along knots $Y_2$ and $K_2$. 
\end{definition}

Let $Y$ be an integer homology three-sphere and let $T$ be a two-dimensional torus embedded in $Y$ in such manner that its normal bundle is trivial. A simple application of the Mayer-Vietoris sequence shows that $Y\setminus N(T)^\circ$ has two connected components $M_1,M_2$, and that each component has the same homology groups as $S^1$. The ``half lives, half dies'' principle shows that for each $i=1,2$ there exists a basis $(\alpha_i,\beta_i)$ for the peripheral subgroup of $\partial M_i$ such that $\beta_i$ is nullhomologous in $M_i$. Therefore, if $Y_i$ denotes the union of $M_i$ and a solid torus $S^1\times D^2$ in such a way that the curve $\{1\}\times\partial D^2$ gets identified with $\alpha_i$, then $Y_i$ is an integer homology three-sphere. Moreover, since $T$ is incompressible in $Y$, then the core of the solid torus in $Y_i$ is a non-trivial knot $K_i$. In other words, every toroidal integer homology three-sphere can be expressed as a splice of non-trivial knots $K_1$ and $K_2$ in integer homology three-spheres $Y_1$ and $Y_2$. 

With all of this in place, we are ready to prove our main result. %, which relies on \Cref{pillowcase alternative} and \Cref{lem:avoids-lines}.

\begin{proof}[Proof of \Cref{thm:toroidal}] Realize $Y$ as a splice $\left(Y_1\setminus N(K_1)^\circ\right)\underset{h}{\cup}\left(Y_2\setminus N(K_2)^\circ\right)$, with $K_1,K_2$ non-trivial knots. Suppose first that $Y_i\setminus N(K_i)^\circ$ is reducible, in other words, that $Y_i\setminus N(K_i)^\circ=Q_i\# \left(Z_i\setminus N(J_i)^\circ\right)$ where $Q_i,Z_i$ are integer homology three-spheres and $J_i\subset Z_i$ has irreducible and boundary-incompressible exterior. As a consequence of Van-Kampen's theorem, there exists a surjection $\pi_1\left(Y_i\setminus N(K_i)^\circ\right)\to \pi_1\left(Z_i\setminus N(J_i)^\circ\right)$, and this surjection induces a $\pi_1$-surjection from $Y$ to the splice of $(Z_1,J_1)$ and $(Z_2,J_2)$. Thus, our proof %that a toroidal integer homology three-sphere $Y$ admits an irreducible $SU(2)$-representation 
reduces to the case when $Y$ is the splice of two knots with irreducible and boundary-incompressible exteriors, which we assume from now on.

Next, by the Seifert--van Kampen theorem, the pieces of the decomposition fit into the following commutative diagram 
$$\xymatrix@R=10pt{
&\pi_1\left(Y_1\setminus N(K_1)^\circ\right)\ar[dr]&\\
\pi_1(T)\ar[ur]\ar[dr]&&\pi_1(Y)\\
&\pi_1\left(Y_2\setminus N(K_2)^\circ\right)\ar[ur]&
}$$ 
and since each $Y_i\setminus N(K_i)^\circ$ is a homology circle, there exists a $\pi_1$-surjection from $Y$ to each $Y_i$. Therefore, our proof reduces further to the case when both $Y_1$ and $Y_2$ are $SU(2)$-cyclic since an irreducible representation for $Y_i$ gives rise to one for $Y$. 

To recap, the previous two paragraphs allow us to assume that $Y$ is the splice of $(Y_1,K_1)$, $(Y_2,K_2)$ with each $Y_i$ an $SU(2)$-cyclic homology three-sphere, and each $K_i\subset Y_i$ a knot with irreducible and boundary-incompressible exterior. Then, as a consequence of \Cref{homologysu2cyclic} we have that each $Y_i$ has trivial instanton Floer homology. Moreover, since each $Y_i\setminus N(K_i)^\circ$ is irreducible and boundary-incompressible, \Cref{instanton-0-surgery} shows that the instanton Floer homology of $0$-surgery on $Y_i$ along $K_i$ is non-zero. Therefore, the hypotheses of both \Cref{pillowcase alternative} and \Cref{lem:avoids-lines} hold, and the proof now follows exactly as in \cite[Proof of Theorem 8.3(i)]{Zentner} with \cite[Theorem 7.1]{Zentner} and \cite[Proposition 8.1(ii)]{Zentner} replaced by \Cref{pillowcase alternative} and \Cref{lem:avoids-lines} respectively.  
\end{proof}

\begin{figure}[h!]
\centering
\def\svgwidth{0.3\textwidth}
%% Creator: Inkscape 1.0 (4035a4f, 2020-05-01), www.inkscape.org
%% Accompanies image file 'pillowcase_splice_trefoils.pdf' (pdf, eps, ps)
\begingroup%
  \makeatletter%
  \providecommand\color[2][]{%
    \errmessage{(Inkscape) Color is used for the text in Inkscape, but the package 'color.sty' is not loaded}%
    \renewcommand\color[2][]{}%
  }%
  \providecommand\transparent[1]{%
    \errmessage{(Inkscape) Transparency is used (non-zero) for the text in Inkscape, but the package 'transparent.sty' is not loaded}%
    \renewcommand\transparent[1]{}%
  }%
  \providecommand\rotatebox[2]{#2}%
  \newcommand*\fsize{\dimexpr\f@size pt\relax}%
  \newcommand*\lineheight[1]{\fontsize{\fsize}{#1\fsize}\selectfont}%
  \ifx\svgwidth\undefined%
    \setlength{\unitlength}{115.33747101bp}%
    \ifx\svgscale\undefined%
      \relax%
    \else%
      \setlength{\unitlength}{\unitlength * \real{\svgscale}}%
    \fi%
  \else%
    \setlength{\unitlength}{\svgwidth}%
  \fi%
  \global\let\svgwidth\undefined%
  \global\let\svgscale\undefined%
  \makeatother%
  \begin{picture}(1,1.63216688)%
    \lineheight{1}%
    \setlength\tabcolsep{0pt}%
    \put(0.60278513,1.03438679){\color[rgb]{0,0,0}\makebox(0,0)[lt]{\lineheight{0}\smash{\begin{tabular}[t]{l} \end{tabular}}}}%
    \put(0,0){\includegraphics[width=\unitlength,page=1]{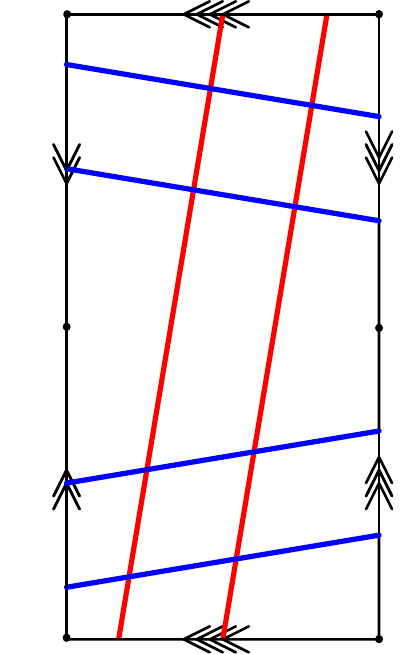}}%
    \put(0.03774589,0.80910789){\color[rgb]{0,0,0}\makebox(0,0)[lt]{\lineheight{0}\smash{\begin{tabular}[t]{l}$P$\end{tabular}}}}%
    \put(1.00731549,0.81073355){\color[rgb]{0,0,0}\makebox(0,0)[lt]{\lineheight{0}\smash{\begin{tabular}[t]{l}$Q$\end{tabular}}}}%
  \end{picture}%
\endgroup%

\caption{Let $Y$ be the three-manifold obtained as the splice of two copies of the exterior of a right handed trefoil, and let $T$ be the incompressible torus given as the intersection of the two knot exteriors. The figure shows the image of each copy of $R(T_{2,3})^*$ in the pillowcase. Note that any representation of the splice corresponding to an intersection of the red and blue curves is irreducible.}
\label{figure:splice_trefoils}
\end{figure}
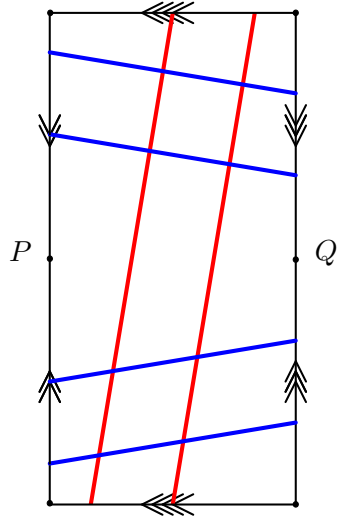

%$$\xymatrix@R=10pt{
%&\pi_1\left(Y_1\setminus N(K_1)\right)\ar[dr]&\\
%\pi_1(T)\ar[ur]^{j_1}\ar[dr]_{j_2\smallcirc h}&&\pi_1(Y)\\
%&\pi_1\left(Y_2\setminus N(K_2)\right)\ar[ur]&
%}$$ 

\section{Review of instanton Floer homology and holonomy perturbations}\label{sec:perturbations}

We start this section with a disclaimer: We do not claim to prove any original or new result in this section. However, we review instanton Floer homology and holonomy perturbations to the extent which is necessary in order to understand the proof of our main results above. For instance, Section \ref{sec:holonomy perturbations and shearing maps} below contains a synthesis of the third author's results about holonomy perturbations from \cite{Zentner} which we hope the reader unfamiliar with this reference will appreciate. Section \ref{sec: invariance} below contains a result about invariance under holonomy perturbations in the context of an admissible bundle with non-trivial second Stiefel-Whitney class, together with a sketch of proof. Again, this result is already contained in \cite{FloerOriginal} and \cite{Donaldson}, but by looking up these references it may not be immediately clear whether these results apply verbatim in our situation. 

The proof of Theorem \ref{thm:pillowcase intersection} relies on a non-vanishing result of an instanton Floer homology group $I^w_{*,\Phi}(Y_0(K))$, computed with suitable perturbation terms $\Phi$ of the Chern-Simons function. We will review the construction of these perturbation terms below, which are built from the holonomy along families of circles, parametrized by embedded surfaces. The critical points of the complex underlying the homology group $I^w_{*,\Phi}(Y_0(K))$ will have a clear interpretation in terms of intersections of the representation variety $R(K)$ with certain deformations of the path given by the straight line $\{ \beta = \pi \}$ in the pillowcase,  resulting as the representation variety of the boundary of the exterior of $K$ in $Y$ as before. 

On the other hand, ~\Cref{instanton-0-surgery} yields a non-vanishing result for $I^w_*(Y_0(K))$, defined in the usual way, and in particular without the above class of perturbation terms.  
We can therefore complete the proof from the fact that the two instanton Floer homology groups, $I^w_*(Y_0(K))$ and $I^w_{*,\Phi}(Y_0(K))$, are isomorphic, and we sketch the proof of this below. 
\begin{remark}
	 In the construction of both $I^w_*(Y_0(K))$ and $I^w_{*,\Phi}(Y_0(K))$ there are typically perturbation terms involved for the sake of transversality. These can be chosen as small as one likes, in a suitable sense. We will omit these auxiliary perturbations from our notation. The perturbations labeled by the terms $\Phi$, however, will have a clear geometric purpose, and the discussion below will focus on these. 
\end{remark}

\subsection{The Chern-Simons function}
For details on the holonomy perturbations we use we refer the reader to Donaldson's book \cite{Donaldson}, Floer's orginal article \cite{floer}, and the third author's article \cite{Zentner}. 

If we deal with an admissible $SO(3)$-bundle $F \to Y$ over a three-manifold $Y$ with second Stiefel-Whitney class $w$, we may suppose that it arises from an $U(2)$-bundle $E \to Y$ as its adjoint bundle $\su(E)$, see for instance \cite[Section 5.6]{Donaldson}. Then $w = w_2(E) \equiv c_1(E) \text{ mod } 2$. The space of $SO(3)$-connections on $F$ is then naturally isomorphic to the space of $U(2)$-connections on $E$ that induce a fixed connection $\theta$ in the determinant line bundle $\det(E)$, which we will suppress from notation. 

When dealing with functoriality properties, it is more accurate to consider $w$ to be an embedded 1-manifold which is Poincar\'e dual to $w_2(E) = w_2(F)$, see \cite{KM_Knot_homology_groups_from_instantons}. 

We will fix a reference connection $A_0$ on $E$ and consider the Chern-Simons function
\begin{equation*}
\begin{split}
	\operatorname{CS}\colon \mathscr{A} & \to \R \\
				A & \mapsto \int_Y \tr (2 a \wedge (F_{A_0})_0 + a \wedge d_{A_0} a + \frac{1}{3} a \wedge [a \wedge a]) \, , 
\end{split}
\end{equation*}
 	defined on the affine space $\mathscr{A}$ of connections $A$ in $E$ which induce $\theta$ in $\det(E)$, and where we have written $A = A_0 + a$ with $a \in \Omega^1(Y;\mathfrak{su}(E))$. The term $F_A$ denotes the curvature of a connection $A$, and $(F_A)_0$ denotes its trace-free part, and $d_{A}$ denotes the exterior derivative associated to a connection $A$.	
% 	Here $(F_{A})_0$ denotes 
% 	the trace-free part of the curvature of the $U(2)$-connection $A$. 
 	We denote by $\mathscr{G}$ the group of bundle automorphisms of $E$ which have determinant $1$. The Chern-Simons function induces a circle-valued function $\overline{\operatorname{CS}}\colon \mathscr{B} \to \R/\Z$ on the space $\mathscr{B} = \mathscr{A}/\mathscr{G}$ of connections modulo gauge equivalence, and 
 	the instanton Floer homology $I^w_*(Y)$ is the Morse homology, in a suitable sense, of the Chern-Simons function $\overline{\operatorname{CS}}$. To carry this out, one has to deal with a suitable grading on the critical points, which will only be a relative $\Z/8$-grading, with suitable compactness arguments (Uhlenbeck compactification and ``energy running down the ends''), and with transversality arguments.  In particular, one will in general add a convenient perturbation term to the Chern-Simons function to obtain the required transversality results. This is usually done by the use of holonomy perturbations that we discuss below. By a Sard-Smale type condition, this term can be chosen as small as one wants, in the respective topologies one is working with.  Therefore, we are suppressing these perturbations for the sake of transversality from our notation. One then needs to prove independence of the various choices involved, and in particular the Riemannian metric and the perturbation terms required for transversality.

 	One may also deal with orientations, but we do not need this in our situation, where $\Z/2$-coefficients in the Floer homology will be sufficient.

 \subsection{Review of holonomy perturbations}\label{sec:hol perturbations}
 	To set up the perturbation of the Chern-Simons function we are using, we need to introduce some notation. 
 	Let $\chi\colon SU(2) \to \R$ be a class function, that is, a smooth conjugation invariant function. Any element in $SU(2)$ is conjugate to a diagonal element, and hence there is a $2 \pi$-periodic even function $g\colon \R \to \R$ such that
\begin{equation}\label{eq:class function}
\chi \left(\begin{bmatrix} e^{it} & 0 \\ 0 & e^{-it} \end{bmatrix} \right) = g(t) \, 
\end{equation}
for all $t \in \R$. Furthermore, let $\Sigma$ be a compact surface with boundary, and let $\mu$ be a real-valued two-form which has compact support in the interior of $\Sigma$ and with $\int_\Sigma \mu = 1$. Let $\iota\colon \Sigma \times S^1 \to Y $ be an embedding. Let $N \subseteq Y$ be a codimension-zero submanifold containing the image of $\iota$, and such that the bundle $E$ is trivialized over $N$ in such a way that the connection $\theta$ in $\det(E)$ induces the trivial product connection in the determinant line bundle of our trivialization of $E$ over $N$. This means that connections in $\mathscr{A}$ can be understood as $SU(2)$-connections in $E$ when restricted to $N$. 
\\

Associated to this data, we can define a function 
\[\Phi\colon \mathscr{A} \to \R\] 
which is invariant under the action of the gauge group $\mathscr{G}$. For $z \in \Sigma$, we denote by $\iota_z\colon S^1 \to Y$ the circle $t \mapsto \iota(z,t)$. A connection $A \in \mathscr{A}$ provides an $SU(2)$-connection over the image of $\iota$. The holonomy $\operatorname{Hol}_{\iota_z}(A)$ of $A$ around the loop $\iota_z$ (with variable starting point) is a section of the bundle of automorphisms of $E$ with determinant $1$ over the loop. Since $\chi$ is a class function, $\chi(\operatorname{Hol}_{\iota_z}(A))$ is well-defined. We can therefore define
\begin{equation} \label{holonomy perturbation term}
	\Phi(A) = \int_{\Sigma} \chi(\operatorname{Hol}_{\iota_z}(A)) \, \mu(z) \, ,
\end{equation}
and this function is invariant under the action of the gauge group $\mathscr{G}$. It depends on the data $(\iota,\chi,\mu)$ and a trivialization of the bundle over a codimension-zero submanifold $N$, but we will omit the latter from notation. 
\\

We will have to work with a finite sequence of such embeddings, all supported in a submanifold $N$ of codimension zero over which the bundle $E \to N$ is trivial. For some $n \in \N$, let 
$
	\iota_k \colon S^1 \times \Sigma_k \to N \subseteq	Y
$
 be a sequence of embeddings for $k=0, \dots, n-1$ such that the interior of the image of $\iota_k$ is disjoint from the interior of the image of $\iota_l$ for $k \neq l$.  We also suppose class functions $\chi_k\colon SU(2) \to \R$ corresponding to even, $2 \pi$-periodic  functions $g_k\colon \R \to \R$ as above to be chosen, for $k = 0, \dots, n-1$, and we assume that $\mu_k$ is a two-form on $\Sigma_k$ with support in the interior of $\Sigma_k$ and integral $1$. Just as in the case of (\ref{holonomy perturbation term}), this data determines a finite sequence of functions 
\[
\Phi_k\colon \mathscr{A} \to \R \, , \hspace{1cm} k = 0, \dots, n-1\, , 
\]	
 	and we are interested in the Morse homology of the function 
 \begin{equation}\label{eq:Chern-Simons perturbed}
 	\overline{\operatorname{CS}} + \Psi \colon \mathscr{B} \to \R/\Z, \quad \text{where} \ \Psi = \sum_{k=0}^{n-1} \Phi_k.
 \end{equation}
 
\begin{definition}\label{def:perturbed CS critical}
	We denote by $R^w_\Psi(Y)$ the space of critical points $[A] \in \A/\G$ of the function $\overline{\operatorname{CS}} + \Psi \colon \mathscr{B} \to \R/\Z$, where $\Psi$ is specified by the holonomy perturbation data $\{\iota_k,\chi_k\}$ as above. 
\end{definition}

If the holonomy perturbation data $\Psi$ is chosen in a way such that $R^w_\Psi(Y)$ does not contain equivalence classes of connections $[A]$ such that $A$ is reducible, then the construction for defining a Floer homology $I^w_\Psi(Y)$ with generators given by critical points of the perturbed Chern-Simons function $\overline{\operatorname{CS}} + \Psi$, and with differentials defined from negative gradient flow lines, goes through in the same way as in \cite{Donaldson,floer}. This will require additional small perturbations in order to make the critical points non-degenerate and in order to obtain transversality for the moduli spaces of flow-lines.  
%\textcolor{red}{TL: Should we say something more clearly that we can make a choice of $\Psi$ so we have non-degeneracy of critical points / transversality for moduli spaces?}  
In fact, we really have not done anything new compared to the constructions in these references since the same perturbations already appear there for the sake of obtaining transversality of the moduli spaces involved in the construction. The only slight difference is that in Floer's work, the surfaces $\Sigma_k$ appearing in the definition of the embeddings $\iota_k$ are always chosen to be disks, whereas those used in the proof of Theorem \ref{thm:pillowcase intersection} above, i.e. in \cite[Theorem 4.2 and Proposition 5.3]{Zentner}, the surfaces $\Sigma_k$ are all annuli. 

More specifically, for completeness, we recall a bit more on the implementation of the holonomy perturbations used in the work of the third author as needed in the previous section for studying $Y_0(K)$.  Given a smoothly embedded path $c$ from $P = (0,\pi)$ to $Q = (\pi,\pi)$ avoiding $(0,0)$ and $(\pi,0)$,  There is an isotopy $\phi_t$ through area-preserving maps of the pillowcase $R(T^2)$ such that $\phi_1$ maps the straight line $c_0:=\{\beta = \pi\}$ from $P$ to $Q$ to the path $c$, and such that $\phi_t$ fixes the four corner points of the pillowcase. Theorem 3.3. of \cite{Zentner} states that isotopies through area preserving maps can be $C^0$-approximated by isotopies through finitely many {\em shearing} maps. For details on shearing maps we refer the reader to \cite[Sections 2 and 3]{Zentner}. The essential relationship is outlined in the following subsection, which we include for the sake of clarity and completeness of our exposition.  

\subsection{Review of holonomy perturbations and shearing maps} \label{sec:holonomy perturbations and shearing maps}
We denote by $R(N)$ the space of flat $SU(2)$-connections in the trivial $SU(2)$-bundle over $N= S^1 \times \Sigma$  up to gauge equivalence, where  $\Sigma = S^1 \times I = S^1 \times [0,1]$ is an annulus. The two inclusion maps $i_-\colon S^1 \times (S^1 \times \{0\}) \to N$ and $i_+\colon S^1 \times (S^1 \times \{1\}) \to N$ induce restriction maps $r_-, r_+ \colon R(N) \to R(T^2)$ to the representation varieties of the two boundary tori, which are pillowcases. In this situation, we have that both $r_-$ and $r_+$ are homeomorphisms, and under the natural identification of these tori we have $r_- = r_+$. 

Now if $\chi$ is a class function as in Equation (\ref{eq:class function}) above then instead of the flatness equation $F_A = 0$ for connections $A$ on the trivial bundle over $N$, one may consider the equation
\begin{equation}\label{eq: perturbed flat}
	F_A = \chi'(\Hol_l(A)) \mu,	
\end{equation}
where $l = S^1 \times {pt}$ denote ``longitudes'' in $N$, where $\Hol_l(A)$ is the holonomy of $A$ along longitudes parametrised by points in $\Sigma$, and where  $\chi'\colon SU(2) \to \su(2)$ is the trace dual of the derivative $d\chi$ of $\chi$, and where $\mu$ is a 2-form with compact support in the interior of $\Sigma$ and $\int_\Sigma \mu = 1$. It can then be proved that $\Hol_l(A)$ does not depend on the choice of longitude, and that solutions $A$ of this equation are reducible, see \cite[Lemma 4]{Braam-Donaldson}, and also \cite[Proposition 2.1]{Zentner}. 

If we denote by $R_\chi(N)$ the solutions of Equation (\ref{eq: perturbed flat}) up to gauge equivalence, then we still have two restriction maps $r_\pm \colon R_\chi(N) \to R(T^2)$. However, in this situation we have the following relationship:
\begin{proposition}
	\label{prop: shearing}
The two restriction maps $r_\pm$ are homeomorphisms and fit into a commutative diagram
\begin{equation}\label{eq: comm diagram shearing}
\begin{split}
\xymatrix{ &R_\chi(N)\ar[dl]_{r_-}\ar[dr]^{r_+}&\\ R(T^2)\ar[rr]^{\phi_\chi} &&R(T^2),} 
\end{split}
\end{equation} 
where $\phi$ is a shearing map that relates to $\chi$ as follows: 

If we write $m_- = \{ pt \} \times S^1 \times \{0\}$ and $m_+ = \{ pt \} \times S^1 \times \{1\}$ for ``meridians'' given by the boundaries of $\Sigma$ in $\{pt \} \times \Sigma$, and if 
\begin{equation}\label{eq: holonomies}
\begin{split}
	\operatorname{Hol}_{m_\pm}(A)  = \begin{bmatrix} e^{i \beta_\pm} & 0 \\ 0 & e^{-i \beta_\pm} \end{bmatrix}, \hspace{0,2cm} \text{and} \hspace{0,2cm} 
	\operatorname{Hol}_{l}(A)  = \begin{bmatrix} e^{i \alpha} & 0 \\ 0 & e^{-i \alpha} \end{bmatrix},
\end{split}
\end{equation}
which we may suppose up to gauge equivalence, then we have 
\begin{equation}\label{eq: shearing}
\begin{split}
	\phi_\chi \left( \begin{array}{c}  \alpha \\ \beta_- \end{array} \right) & = 
		\left( \begin{matrix}  \alpha    \\ \beta_- + f(\alpha)  \end{matrix} \right),  
\end{split} 
\end{equation}
where $f\colon \R \to \R$ is the derivative of the function $g$ appearing in Equation (\ref{eq:class function}). 
Here, $(\alpha,\beta_\pm)$ determine points in $R(T^2)$ determined by $\Hol_{m_\pm}(A)$ and $\Hol_{l}(A)$ as in Equation (\ref{eq: holonomies}) above. 
\end{proposition}

Equation (\ref{eq: comm diagram shearing}) is essentially proved in \cite[Lemma 4]{Braam-Donaldson}, and a proof also appears in \cite[Proposition 2.1]{Zentner}. 

Of course, one can iterate this construction: One may choose a finite collection of disjoint embeddings $\iota_k\colon S^1 \times \Sigma$ into a closed three-manifold $Y$, and class functions $\chi_k$. The embeddings may chosen to be ``parallel'' in that the image of $\iota_k$ corresponds to $ S^1 \times (S^1 \times [k,k+1]) \subseteq  S^1 \times (S^1 \times [0,n]) \subseteq Y$, but the role of ``meridian'' and ``longitude'' may be chosen arbitrarily in an $SL_2(\Z)$ worth of possible choices. In this case the restriction maps to the two boundary components of $S^1 \times (S^1 \times [0,n])$ in the diagram analogous to Equation (\ref{eq: comm diagram shearing}) will be related by a {\em composition} of shearing maps.

\subsection{Holonomy perturbations and the pillowcase}

  The main application of holonomy perturbations we have in mind is stated as Theorem \ref{thm: hol perturbations, missing curve} below. To put it into context, note first that for a non-trivial bundle, the critical space of the Chern-Simons function $R^w(Y_0(K))$ is a double cover of $R(K|c_0)$, where $c_0$ is the straight line from $(0,\pi)$ to $(\pi,\pi)$ in the pillowcase, see \cite[Proposition 5.1]{Zentner}. If we choose holonomy perturbations associated to some data $\{ \iota_k, \chi_k\}_{k=0}^{n-1}$ as above, where the image of $\iota_k$ corresponds to $S^1 \times (S^1 \times [k,k+1]) \subseteq  S^1 \times (S^1 \times [0,n]) \subseteq Y$ in a collar neighborhood of the Dehn filling torus in $Y_0(K)$, then repeated use of \Cref{prop: shearing} above will imply that for the holonomy perturbation $\Psi$ determined by the data $\{ \iota_k, \chi_k\}_{k=0}^{n-1}$, the critical space of $R^w_\Psi(Y_0(K))$ will correspond to $R(K|c')$, where $c'$ is the image of $c_0$ under a composition of shearing maps $\phi_{n-1} \circ \dots \circ \phi_0$, with ``directions'' determined by the embeddings $\iota_k$. 
 	(In Equation (\ref{eq: shearing}) we are dealing with a shearing in direction $\begin{pmatrix} 0 \\ 1 \end{pmatrix}$, but we can pick any direction in $\Z^2$.)
  
  	The main point of \cite[Theorem 4.2]{Zentner} is that the area-preserving maps of the pillowcase obtained by composition of shearing maps is $C^0$-dense in the space of all area-preserving maps of the pillowcase, and this yields the following result. 
  
%  There exist holonomy perturbations associated to the data $\{ \iota_k, \chi_k\}_{k=0}^{n-1}$, determined by the isotopy $\phi_t$ above, and where the images of the embeddings $\iota_{k}\colon S^1 \times (S^1 \times I) \to Y_0(K)$ are supported in a collar neighborhood of the boundary of the exterior of $K$.    We summarize this with the following theorem (adapted from knots in $S^3$ to arbitrary $SU(2)$ cyclic integer homology three-spheres):  %The construction is completely analogous to \cite[Section 5]{Zentner} with the only difference of replacing $S^3$ therein with the more general $SU(2)$-cyclic integer homology 3-sphere $Y$. 
	
\begin{theorem}[Theorem 4.2 and Proposition 5.3, \cite{Zentner}]\label{thm: hol perturbations, missing curve}
Let $K$ be a knot in an $SU(2)$-cyclic integer homology three-sphere $Y$.  Let $c$ be an embedded path from $(0,\pi)$ to $(\pi,\pi)$ missing the other orbifold points of the pillowcase.  Then, there exists an arbitrarily close path $c'$ and a holonomy perturbation $\Psi$ along disjoint embeddings of $S^1 \times (S^1 \times I)$ parallel to the boundary of a neighborhood of $K$ such that $R^w_\Psi(Y_0(K))$ double-covers $R(K | c')$.  
\end{theorem}
(To see that we get a double-cover here we refer the reader to \cite[Remark 1.2]{Zentner}).

We only stress the fact that we must assume there are no reducible connections in $R^w_\Psi(Y)$, since the presence of such solutions will result in a failure of the transversality arguments involved in the discussion.

%The following is essentially proved in \cite{Braam-Donaldson}.
%\begin{proposition}\label{prop:critical_points_perturbed_CS}
%  The critical points of the perturbed Chern-Simons function
%  \[
%  	\operatorname{CS} + \sum_{k = 0}^{n-1} \Phi_k \colon \mathscr{A} \to \R
%  \]
%are the elements $A \in \mathscr{A}$ which solve the equation
%\begin{equation}\label{perturbed flatness}
%	(F_A)_0 = \sum_{k=0}^{n-1} \chi'_k (\operatorname{Hol}_{\iota_k}(A)) \, \mu_k \, , 
%\end{equation}
%where $\chi'_k$ is an equivariant map $SU(2) \to \su(2)$ which is the dual to the derivative of $\chi_k$ with respect to the Killing form on $\su(2)$, and where the pull-back of $\mu_k$ by $\iota_k$ is the 2-form on $S^1 \times \Sigma$ which is obtained by pulling back $\mu$ from $\Sigma$ to $S^1 \times \Sigma$.   
%\end{proposition}
%
%
%\begin{definition}
%	For choices made as above, we denote 
%	\[
%		R^{w}_{\{ \iota_k, \chi_k \}}(Y) =  \{ [A] \in \mathscr{A}/\mathscr{G} \, | \, (F_A)_0 = \sum_{k=0}^{n-1} \chi'_k (\operatorname{Hol}_{\iota_k}(A)) \, \mu_k \, \} \, ,
%	\] 
%	and we call this the perturbed $SO(3)$-representation variety of $Y$ associated to $w$ and the holonomy perturbation data $\{ \iota_k, \chi_k \}$. 
%\end{definition}
%\begin{remark}
%	This definition depends also on the choice of a trivialisation of the bundle $E$ over some codimension-0 submanifold which contains the images of the $\iota_k$. 
%\end{remark}
%
%\begin{remark}
%A similar statement as in Remark \ref{re:hol correspondence} above applies to the perturbed representation variety $R^w_{\{ \iota_k, \chi_k \}}(Y)$. 
%\end{remark}

\subsection{Invariance of instanton Floer homology}\label{sec: invariance}
	The instanton Floer homology groups $I^w(Y)$ and $I^w_\Psi(Y)$, the latter being defined under the additional assumption that $R^w_\Psi(Y)$ does not contain reducible connections, depend on additional data that we have already suppressed from notation, notably the choice of a Riemannian metric on $Y$ and holonomy perturbations just as defined above in order to achieve transversality. More explicitly, holonomy perturbations have already been implicit in the definition of instanton Floer homology unless the critical points of $\operatorname{CS}$ had been non-degenerate at the start and the moduli space defining the flow lines had been cut out transversally.  In Floer's original work \cite{floer}, and elaborated in more detail in Donaldson's book \cite{Donaldson}, invariance under the choice of Riemannian metric and the choice of holonomy perturbations follows from a more general concept, namely the functoriality of instanton Floer homology under cobordisms. See also the discussion in \cite[Section 3.8.]{KM_Knot_homology_groups_from_instantons}
	
%	We consider a general class of perturbations to the Chern-Simons function for which we make only few a priori assumptions. They should all have the property that the linearization of the perturbed Chern-Simons function is the same as the linearization of the unperturbed one, up to addition of a compact operator. Maybe define a class of admissible holonomy perturbations. \dots 

%	\begin{definition}
%		\label{def:critical_perturbed-chern-simons}
%		We denote by $R^w_{\Phi}(Y)$ the set of critical points of the perturbed Chern-Simons function 
%		\begin{equation*}
%	\begin{split}
%	\operatorname{CS} + \Phi \colon \mathscr{A} & \to \R\, ,
%	\end{split}
%	\end{equation*}
%	up to gauge equivalence, and we call this the associated perturbed representation variety. 
%	\end{definition}
	
	\begin{theorem}[Invariance under holonomy perturbations]\label{thm:invariance of instanton Floer homology}
			Suppose that the space of critical points $R^w_\Psi(Y)$ of the perturbed Chern-Simons function \ref{eq:Chern-Simons perturbed} appearing in Definition \ref{def:perturbed CS critical} above does not contain equivalence classes of reducible connections.  
		Then the associated instanton Floer homology groups $I^w_*(Y)$ and $I^w_{*,\Psi}(Y)$ are isomorphic. 
	\end{theorem}
	
\begin{proof}[Sketch of Proof]
The proof of this statement is standard, so we %	The proof of this statement is ``standard'' but not at all trivial if one intends to write down all technical details involved (and they really don't appear fully written down anywhere, hahaha!)\marginpar{remove this comment later}. 
will describe a chain map determining the isomorphism on homology and outline the ideas along which the result is proved.  
	
	Slightly more generally, suppose we are dealing with a smooth map $[0,1] \to C^\infty(\mathscr{A},\R)$, $s \mapsto\Gamma(s)$. We may suppose that this map is constant near $0$ and $1$. The Floer differential counts flow lines of the Chern-Simons function, possibly suitably perturbed. Instead of doing this, we may also consider the downward gradient flow equation of the {\em time-dependent} function $\operatorname{CS} + \Gamma(s)$, where we extend $\Gamma(s)$ to a map $(-\infty,\infty) \to  C^\infty(\mathscr{A},\R)$ which is constant $\Gamma(0)$ on $(-\infty,0]$ and constant $\Gamma(1)$ on $[1,\infty)$.  If we are given critical points $\rho_0$ of $\operatorname{CS} + \Gamma(0)$ and $\rho_1$ of $\operatorname{CS} + \Gamma(1)$ of the same index, then we consider a zero-dimensional moduli space $M_{\rho_0,\rho_1}$ of connections ${\bf{A}} = \{A(t)\}_t$ on $E \to \R \times Y$ of finite $L^2$-norm (inducing $\theta$ on $\det(E)$, pulled back to $\R \times Y$), such that the equation 
	\begin{equation}\label{eq:perturbed_grad_flow}
		\frac{d A}{dt} = - \operatorname{grad}(	\operatorname{CS} + \Gamma(t))(A(t))
	\end{equation}
		holds on $\R \times Y$, where $\operatorname{grad}$ denots the $L^2$-gradient, and $\bf{A}$ limits to $\rho_0$ and $\rho_1$ in the limit $t \to \pm \infty$, respectively.  Finally, we also require that the moduli space $M_{\rho_0,\rho_1}$ is cut out transversally.
	
	We require that the addition of the term $-\operatorname{grad}(\Gamma(t))(A(t))$ to the gradient flow \Cref{eq:perturbed_grad_flow} for the Chern-Simons function does not alter the linearized deformation theory for $\bf{A}$, see for instance \cite[Sections 3 and 4]{Donaldson}. Furthermore, we have to require that the Uhlenbeck compactification goes through with the perturbation we have in mind.  It is shown in \cite[Section 5.5]{Donaldson} that both hold for the function $\Gamma$ built from holonomy perturbations as described in \Cref{eq:Gamma}.  One essential feature is that the holonomy perturbation term appearing in the flow equation is uniformly bounded. 
	
	 A suitable interpolation between the holonomy perturbation data $\Gamma(0) = 0 $ and $\Gamma(1) = \Psi$ for $\Psi$ as in Equation (\ref{eq:Chern-Simons perturbed}) is given, for instance, by the following formula. Suppose $\Psi$ is determined by data $\{\iota_k,\chi_k\}_{k=0}^{n-1}$. Then for $t \in [\frac{k}{n},\frac{k+1}{n}]$ we define
\begin{equation}\label{eq:Gamma}
			\Gamma(t) = \sum_{l=0}^{k-1} \Phi_l + \beta(t-k/n) \Phi_k			
	\end{equation}
		for any $k \in \{0,\dots, n-1\}$. Here $\beta\colon[0,\frac{1}{n}] \to [0,1]$ is a smooth function which is $0$ in a neighborhood of $0$ and $1$ in a neighborhood of $\frac{1}{n}$.  
		
		Now the moduli space $M_{\rho_0,\rho_1}$ does not contain any reducibles, because if it did, then the limits $\rho_0$ and $\rho_1$ in $R^w(Y)$ and $R^w_\Psi(Y)$, respectively, would also be reducible, and by our assumption and the setup for instanton Floer homology for admissible bundles, this does not occur. 
		
		One defines a linear map $\zeta\colon C^w(Y) \to C^w_\Psi(Y)$ of the underlying chain complexes such that the ``matrix entry'' corresponding to the elements $\rho_0 \in C^w(Y)$ and $\rho_1 \in C^w_\Psi(Y)$ is given by the signed count of the moduli space $M_{\rho_0,\rho_1}$, where the sign is determined in the usual way by the choice of a {\em homology orientation}. That $\zeta$ is a chain map follows from analyzing the compactification of suitable $1$-dimensional moduli spaces, making use of Uhlenbeck compactification -- no bubbling can occur here due to the dimension of the moduli space -- and the chain convergence discussed in \cite[Section 5.1]{Donaldson}, together with suitable glueing results. 
		
		That different interpolations yield chain homotopic chain maps follows from studying the compactification of $(-1)$-dimensional moduli spaces over a 1-dimensional family, defining a chain homotopy equivalence between the two different interpolations. 
		
		That $\zeta$ defines a chain homotopy equivalence follows from the functoriality property: One may consider a further path $\Gamma'\colon [1,2] \to  C^\infty(\mathscr{A},\R)$ such that $\Gamma'(1) = \Gamma(1)$, similar as above. This defines a corresponding chain map $\zeta'\colon C^w_{\Psi} \to C^w_{\Gamma'(2)}$. On the other hand, one may concatenate the path $\Gamma(t)$ and the path $\Gamma'(t)$ and build a corresponding interpolation $\Gamma''\colon [0,2] \to C^\infty(\mathscr{A},\R)$, resulting in a chain map $\zeta''$ as above. A neck stretching argument then shows that $\zeta''$ and $\zeta' \circ \zeta$ are chain homotopy equivalent, and hence induce the same maps on homology. In our situation we take $\Gamma'(2)$ to be $0$, meaning that this defines again the ``unperturbed'' chain complex $C^w(Y)$ (which, again, may contain some perturbations for the sake of regularity omitted in our notation). One may finally interpolate between $\Gamma''$ and the $0$-term along a 1-dimensional family. Analyzing again the compactification of suitable $(-1)$-dimensional moduli spaces over a 1-dimensional family, one obtains a chain homotopy equivalence between $\zeta''$ and the identity. 		
\end{proof}
	
\begin{remark}
	There is some confusion about invariance under ``small'' and ``large'' holonomy perturbations in the field. If one is given holonomy perturbation data for which the underlying space of critical points and moduli spaces defining the differentials are already cut out transversally, then for small enough perturbations the same will still hold, and the resulting chain complexes will be isomorphic. This is due to the fact that the condition of being cut out transversally is an {\em open} condition, expressed as the surjectivity of the deformation operators involved in the linearized equation together with the Coulomb gauge fixing. 
	
	If on the other hand, one is given a situation where the critical points and the unperturbed moduli spaces are not cut out transversally, then one needs to perturb, and even if these perturbations are chosen ``small'', the resulting chain complexes will in general not be isomorphic but only chain homotopy equivalent. In this situation, the proof of invariance is really the same as proving the invariance under ``large'' perturbations, and already present in \cite{floer} and \cite{Donaldson}. 
\end{remark}

%In Section~\ref{sec:pillowcase} below, we will need to compare the instanton Floer homologies for admissible bundles computed two different types of holonomy perturbations: 
%\begin{itemize}
%\item A small holonomy perturbation to the Chern-Simons functional as originally defined by Floer
%\item A large holonomy perturbation to the Chern-Simons functional supported on $T^2 \times I$ to identify intersections of curves in the pillowcase with solutions to the perturbed flat equations
%\end{itemize}

%
%The following is well-known: 
%\begin{proposition}\label{prop:perturbed-floer-homology}
%Let $(N,w)$ be an admissible pair.  Let $\pi$ be an arbitrary holonomy perturbation of the Chern-Simons functional such that the critical points are non-degenerate and irreducible, and the moduli spaces are regular.  Then, the associated instanton Floer homology, $I^{w,\pi}_*(N)$, is isomorphic to $I^{w}_*(N)$.  
%\end{proposition}
%\begin{proof}
%This is standard and we may or may not include a proof.  
%\end{proof}

%\section{Proof of the Theorem \ref{thm:pillowcase intersection}}\label{sec:proof_main}

\section{Branched covers of prime satellite knots}\label{sec:other}
In this section, we prove \Cref{cor:branched-cover}, establishing the existence of a non-trivial $SU(2)$ representation for cyclic branched covers of prime satellite knots. We begin with a definition of satellite knots.

\begin{definition} \label{def:satellite}
Let $P\subset S^1\times D^2$ be an oriented knot in the solid torus. Consider an orientation-preserving embedding $h \colon S^1\times D^2\to S^3$ whose image is a tubular neighborhood of a knot $K$ so that  $S^1\times \{*\in \partial D^2\}$ is mapped to the canonical longitude of $K$. The knot $h(P)$ is called the {\em satellite knot with pattern $P$} and {\em companion $K$}, and is  denoted $P(K)$. The {\em winding number} of the satellite is defined to be the algebraic intersection number of $P$ with  $\{*\}\times D^2$.  See \Cref{fig:satellite} for an example. \end{definition}

\begin{figure}
    \centering
    \def\svgwidth{0.2\textwidth}
    \input{21cablefigure}\hspace{1cm}
     \def\svgwidth{0.3\textwidth}
    %% Creator: Inkscape 1.0beta1 (32d4812, 2019-09-19), www.inkscape.org
%% Accompanies image file '(2,1)-cable_trefoil.pdf' (pdf, eps, ps)
\begingroup%
  \makeatletter%
  \providecommand\color[2][]{%
    \errmessage{(Inkscape) Color is used for the text in Inkscape, but the package 'color.sty' is not loaded}%
    \renewcommand\color[2][]{}%
  }%
  \providecommand\transparent[1]{%
    \errmessage{(Inkscape) Transparency is used (non-zero) for the text in Inkscape, but the package 'transparent.sty' is not loaded}%
    \renewcommand\transparent[1]{}%
  }%
  \providecommand\rotatebox[2]{#2}%
  \newcommand*\fsize{\dimexpr\f@size pt\relax}%
  \newcommand*\lineheight[1]{\fontsize{\fsize}{#1\fsize}\selectfont}%
  \ifx\svgwidth\undefined%
    \setlength{\unitlength}{241.50773621bp}%
    \ifx\svgscale\undefined%
      \relax%
    \else%
      \setlength{\unitlength}{\unitlength * \real{\svgscale}}%
    \fi%
  \else%
    \setlength{\unitlength}{\svgwidth}%
  \fi%
  \global\let\svgwidth\undefined%
  \global\let\svgscale\undefined%
  \makeatother%
  \begin{picture}(1,0.85564738)%
    \lineheight{1}%
    \setlength\tabcolsep{0pt}%
    \put(0,0){\includegraphics[width=\unitlength,page=1]{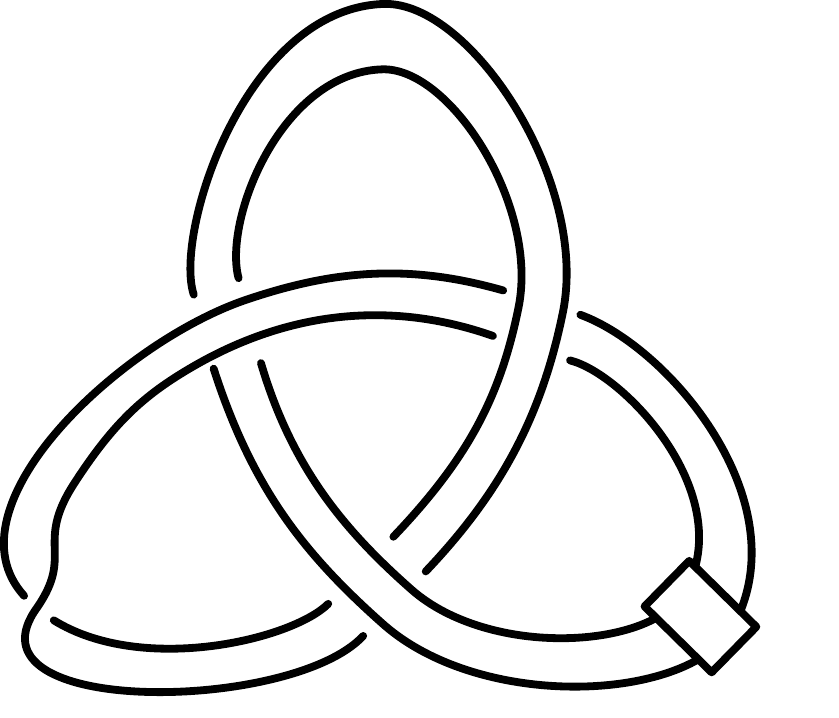}}%
    \put(0.78457429,0.12925047){\color[rgb]{0,0,0}\rotatebox{-46.174991}{\makebox(0,0)[lt]{\lineheight{1.25}\smash{\begin{tabular}[t]{l}-3\end{tabular}}}}}%
  \end{picture}%
\endgroup%

    \caption{Left: The pattern representing the (2,1)-cable satellite operation. Right: The (2,1)-cable for the right handed trefoil. The extra twisting appears as a consequence of the requirement that a longitude in $S^1\times D^2$ maps to the canonical longitude of the trefoil.}\label{fig:satellite}
\end{figure}

\begin{repcorollary}{cor:branched-cover} Let $K$ be a prime, satellite knot in $S^3$ and let $\Sigma(K)$ be any non-trivial cyclic cover of $S^3$ branched over $K$. Then $\pi_1\left(\Sigma(K)\right)$ admits a non-trivial $SU(2)$ representation. 
\end{repcorollary}

\begin{proof}%[Proof of \Cref{cor:branched-cover}]
Let $K$ be a prime satellite knot in $S^3$. If $\Sigma(K)$ is not an integer homology three-sphere, then there is a non-trivial abelian representation. In the case when $\Sigma(K)$ is an integer homology three-sphere, then by \Cref{thm:toroidal} it suffices to show that $\Sigma(K)$ is toroidal. Write $K = P(J)$ and observe that if $\Sigma(K)$ is the $d$-fold cover of $S^3$ branched over $P(J)$, then there is a decomposition of $\Sigma\left(P(J)\right)$ as the union of $\Sigma(S^1\times D^2, P)$, the $d$-fold cover of $S^1\times D^2$ branched over $P$, and a $d$-fold covering space of the knot complement $S^3\setminus N(J)$. The isomorphism type of this latter covering space depends only on the greatest common divisor between $d$ and the winding number of $S^1\times D^2$, see for example \cite{seifert} or \cite[pg. 220]{livingston-melvin}. Since the exterior of $J$ has incompressible boundary, the same is true of any cover.  Therefore, we just need to show that $\Sigma(D^2 \times S^1, P)$ has incompressible boundary.  We claim the following.  Let $P$ be a non-trivial pattern knot in $D^2 \times S^1$ which does not correspond to a connect-sum and which is not contained in an embedded $B^3$.  Then for any cyclic branched cover over $P$, $\Sigma(D^2 \times S^1, P)$ has incompressible boundary.   This claim is standard and proved in the lemma below for completeness. 
\end{proof}

\begin{lemma}
Let $P$ be a non-trivial pattern knot in $D^2 \times S^1$ which does not correspond to a connect-sum and which is not contained in an embedded $B^3$.  Then for any cyclic branched cover over $P$, $M = \Sigma(D^2 \times S^1, P)$ has incompressible boundary.  
\end{lemma}
\begin{proof}
Suppose that $\gamma$ is an essential loop on $\partial M$, which is nullhomotopic in $M$.  Let $G$ denote the group of covering transformations of $M$ and consider the action of $G$ on the boundary.  We first claim that $\gamma$ can be isotoped on the boundary such that for each $g \in G$, either $g(\gamma) \cap \gamma = \emptyset$ or $g(\gamma) = \gamma$.   Of course, we only need to restrict to the subgroup of elements which fix the boundary component containing $\gamma$ setwise.  Further, since $\gamma$ bounds in $M$, it is easy to see that the homology class in $\partial M$ is fixed by all such elements.  Because every finite group action on the torus is equivalent to the quotient of an affine action of the plane, the claim easily follows.  Now, because of this claim, and because the curve $\gamma$ is disjoint from the lift of $P$, the equivariant Dehn's lemma \cite{MeeksYau} implies that there exists a disk $D$ in $M$ bounding $\gamma$ such that for all $g$, either $g(D) \cap D = \emptyset$ or $g(D) = D$, and furthermore, $D$ is transverse to the lift of the branch set.  Consider the (possibly disconnected) surface $\Sigma = \bigcup_{g \in G} g(D)$.  Then, $\Sigma/G$ is a collection of disks in $D^2 \times S^1$ and $\Sigma \to \Sigma/G$ is a branched cover (although some components of $\Sigma$ may have trivial branch locus).  Furthermore, each component of the boundary of $\Sigma/G$ is an essential curve on the boundary of the solid torus.   For homology reasons, it is necessarily a meridional curve on the solid torus and each component of $\Sigma/G$ is a meridional disk.  (The components cannot have any other topology, since a disk can only cover/branch cover another disk.)  Now, if any component of $\Sigma/G$ does not intersect $P$, then we can cut $D^2 \times S^1$ along one of these disks, and see that $P$ is contained in $B^3$ and we have a contradiction.  If some component of $\Sigma/G$ does intersect $P$, it must intersect in exactly one point, since a disk cannot be such a cyclic branched cover over a disk with more than one branch point. (Here we are using that the branch points all correspond to intersections of $P$ with the disk.)    In other words, $P$ is the pattern for a connect-sum, and again we have a contradiction.  This proves the claim and completes the proof of the lemma.              
\end{proof}

%\begin{proof}[of \Cref{thm:toroidal}]
%\begin{itemize}
%\item Assume that $Y$ is irreducible.
%\item Write $Y$ as a splice of $(Y_1,K_1)$ and $(Y_2,K_2)$ where the exterior of $K_i$ is irreducible and boundary irreducible.
%\item If one of $Y_i$ has an irreducible $SU(2)$ representation, then pull back this representation to $Y$ under the pinch map.
%\item If no $Y_i$ has an irreducible $SU(2)$ representation, prove that $0$-surgery on $K_i$ has non-trivial instanton Floer homology.
%\item Use the pillowcase argument with $Y_i$ in place of $S^3$.  
%\end{itemize}
%\end{proof}

%\textcolor{red}{RZ: In the bibliography it shouldn't be the case that some authors are given their full first name, and others aren't}

\bibliographystyle{abbrv}
%\nocite{*}
\bibliography{SU2toroidalreferences}

\end{document}